\documentclass[12pt, reqno]{amsart}
\usepackage[english]{babel}
\usepackage{amsthm,verbatim} 
\usepackage{amssymb}
\usepackage{amsmath}   
\usepackage{amsfonts}
\usepackage{tikz}
\usepackage{tikz-cd}
\usepackage{graphicx,color}
\usepackage{mathrsfs}  
\usetikzlibrary{math}
\usepackage{hyperref}
\usepackage{todonotes}
\usepackage{enumitem}
\usepackage{soul}
\usepackage[nocompress]{cite}
\usepackage{stmaryrd}
\usepackage{fancyhdr}
\usepackage{caption}
\usepackage{subcaption}
\usepackage{float}
\usepackage[left=1.1in, right=1.1in, top=1.2in, bottom=1.2in]{geometry}
\newtheorem{theorem}{Theorem}[section]
\newtheorem{prop}[theorem]{Proposition}

\newtheorem{lemma}[theorem]{Lemma}
\newtheorem{proposition}[theorem]{Proposition}

\newtheorem{note}[theorem]{Note}

\newtheorem{assumption}{Assumption}
\theoremstyle{definition}

\newtheorem{Example}[theorem]{Example}

\theoremstyle{remark}

\newcommand{\etalchar}[1]{$^{#1}$}

\newcommand{\RNum}[1]{\uppercase\expandafter{\romannumeral #1\relax}}

\newcommand{\cL}{\mathcal L}

\newcommand{\eps}{\varepsilon}
\newcommand{\cQ}{\mathcal Q}

\newcommand{\pa}{\partial}

\newcommand{\be}{\begin{equation}}
\newcommand{\ee}{\end{equation}}

\newcommand{\R}{\mathbb R}
\newcommand{\bE}{\mathbf E}

\newcommand{\cX}{\mathcal{X}}
\newcommand{\cN}{\mathcal{N}}

\newcommand{\xa}{X_A}
\newcommand{\xg}{X_G}
\newcommand{\xp}{X_P}
\newcommand{\col}{^{[i]}}
\newcommand{\eep}{^{\varepsilon}}

\begin{document}

	\title[Coarse graining for SDEs]{Coarse graining of stochastic differential equations: averaging and projection method}
    
    \author[H.~Duong]{Hong Duong$^{(1)}$}
	\address{$^{(1)}$School of Mathematics, University of Birmingham, 
Birmingham B15 2TT, UK} 
\email{h.duong@bham.ac.uk}

\author[C.~Hartmann]{Carsten Hartmann$^{(2)}$}
\address{$^{(2)}$Institut f\"ur Mathematik, Brandenburgische Technische Universit\"at  Cottbus-Senftenberg, 03046 Cottbus, Germany} 
\email{hartmanc@b-tu.de}

\author[M. Ottobre]{Michela Ottobre$^{(3)}$}
\address{$^{(3)}$Maxwell Institute for Mathematical Sciences and Mathematics Department, Heriot-Watt University, Edinburgh EH14 4AS, UK}
\email{m.ottobre@hw.ac.uk}

\begin{abstract}
 We study coarse-graining methods for stochastic differential equations. In particular we consider averaging and a type of projection operator method, sometimes referred to as effective dynamic via conditional expectations. The projection method (PM) we consider is related to the ``mimicking marginals'' coarse graining approach proposed by Gy\"ongy.   The first contribution of this paper is to provide further theoretical background for the PM and a rigorous link to the Gy\"ongy method. Moreover, we compare  PM  and averaging.   While averaging applies to systems with time scale separation, the PM can in principle be applied irrespective of this. However it is often assumed that the two methods coincide in presence of scale separation. The second contribution of this paper is to make this statement precise, provide sufficient conditions under which these two methods coincide and then show -- via examples and counterexamples -- that this needs not be the case. 

\vspace{5pt}
    {\sc Keywords.} Coarse Graining for Stochastic Differential Equations, two-parameter Markov Semigroups, Averaging, Projection Methods.
    
\vspace{5pt}
{\sc AMS Classification (MSC 2020). 60H10, 34F05,60J60, 34K33, 35B40, 82C31.} 
 
\end{abstract}

 \maketitle

\section{Introduction} \label{sec:intro}

Coarse-graining of nonlinear dynamical systems, also called model (order) reduction,  often involves a combination of systematic mathematical approaches and heuristics based on either physical intuition or observations drawn from data.  Prime examples of this fact are projection operator methods,  among which are  optimal prediction (e.g.~\cite{chorin2000optimal}) or stochastic mode reduction approaches (e.g. in climate modelling \cite{majda2001mathematical}) that eliminate small-scale variables from a potentially high-dimensional system.  Other coarse-graining approaches include multiscale methods, prominently averaging and homogenization,  \cite{pavliotis2008multiscale}, thermodynamic (mean field) limits \cite{sznitman1991topics}, Mori-Zwanzig \cite{kupferman2004fractional, mori1965continued},  with this list being by far not exhaustive.

The general purpose of all such methods is to reduce a high dimensional system to a lower dimensional one, in such a way that the latter preserves certain properties of the original dynamics. The eliminated degrees of freedom are often referred to as unresolved variables.  For many of the available methods the starting point is the modelling of the entire high dimensional system; this involves first a modelling step in which the unresolved variables are parametrised by e.g.~a stochastic process or a probability measure, depending on the context, and then, secondly, the elimination step per se. The first step is an issue in its own right;  for example, in the context of slow-fast systems, identifying suitable dimensionless small parameters and modelling the fast variables can be a big challenge \cite{izaguirre2009adaptive}. In this paper we do not consider this first step but rather assume that a model for the entire system is given and focus solely on the elimination, i.e.~the coarse graining step.\footnote{In this paper we use the terms dimension reduction and coarse graining interchangeably.}  In particular we will be concerned with coarse graining of systems modelled by Stochastic Differential Equations (SDEs).  

Coarse graining is needed and routinely used in molecular dynamics, continuum mechanics, control engineering, atmosphere and ocean dynamics, and in many other fields of applied science. Yet some dimension reduction methods have by now undergone complete mathematization, others less so. Moreover, when more than one method is applicable to the system at hand, rigorous comparisions are often still lacking but certainly needed, as advocated e.g.~in  \cite{givon2004extracting}.   Since appropriate mathematization can contribute towards comparison, the two issues are related.  

In this paper we consider two types of coarse graining methods and compare them.  Specifically, we consider the classical averaging principle, which assumes time scale separation between variables (e.g. \cite{khasminskij1968averaging,veretennikov1991averaging}),  and projection operator methods for SDEs in the spirit of the works \cite{legoll2010effective,legoll2017pathwise}. The latter are an approximation of the dimension reduction approach introduced by Gy\"ongy in his seminal paper \cite{gyongy1986mimicking}.  In particular, we address the question under which conditions the projection method reduces to the averaging principle. We motivate  addressing this question in some depth in Subsection \ref{subsec:comparison} and more briefly below.

There are various  projection operator type approaches in the literature that are closely related to the one we will consider, such as  optimal prediction \cite{chorin2000optimal}, stochastic modelling \cite{just2001stochastic},  Born-Markov \cite{moy1999born}, or Born-Oppenheimer approximation \cite{panati2007time}. There are subtle differences between the aforementioned methods though and, for the sake of clarity, we will recall the method presented in \cite{legoll2010effective},  in Section \ref{sec: Main Results}. This is the one we will study here and that from now on we refer to as the Projection Method (PM).

For now we just briefly mention that any projection operator-based method can 
be understood as a form of best approximation in a weighted $L^2$ space where the weight is a probability distribution of choice. In the PM one takes a conditional probability distribution as weight, namely the distribution of the entire invariant measure of the system, conditional on knowledge of the resolved variables. The resulting  projection operator is a conditional expectation.  

While averaging methods for SDEs are by now well grounded on rigorous mathematical foundations, the theory underpinning the PM is less developed (though there are results in the literature, see e.g. \cite{legoll2011some,hartmann2020coarse}). %Moreover, the relation between the PM and the Gy\"ongy approach is only informally justified.  
The first goal of this paper is to contribute to the development of the mathematical understanding of the PM and to make the relation between the PM and the Gy\"ongy approach rigorous, which helps to shed light on potential and limitations of the PM. We show that, for a large class of SDEs, the dynamics produced via the GM and the one produced via the PM have the same long time behaviour.   We moreover  provide sufficient conditions under which the PM works (i.e. conditions under which the produced coarse grained dynamics has the desired properties) and examples where it fails.   

The PM can in principle be used irrespective of whether there is a time-scale separation in the system (by which we mean that the high dimensional system is a slow-fast system, with the unresolved variables being the fast part of the dynamics). However, it is a common belief (e.g.~\cite{chorin2006prediction,kantz2004fast}) that, in presence of scale separation,  the coarse-grained dynamics obtained by PM and the one obtained via averaging coincide (we will be more precise on this in Section \ref{sec: Main Results}). This belief seems to stem from  the fact that when the original system is reversible (in gradient form) then it is (to an extent, which we clarify in the next section) true that the PM approximation boils down to the approximation produced via averaging. 
The second main goal of this paper is to point out that, in general, it is {\em not} the case that these two coarse graining procedures coincide in presence of scale separation. We provide both counterexamples (see Section \ref{sec: Main Results} and Section \ref{sec: rhox and rhox|y}) and sufficient
conditions under which this statement holds.

We note in passing that it is moreover often argued that the so-called \emph{Mori--Zwanzig projection operator formalism} is an extension of the optimal prediction projection operator formalism that can account for lack of scale separation \cite{chorin2002optimal}. Nevertheless, it is not even clear, beyond special cases (e.g. \cite{kupferman2004fractional}), whether high-fidelity non-Markovian equations can be obtained from the Mori--Zwanzig coarse-graining procedure when no scale separation is present; cf.~\cite{zhang2016effective,lin2018milestoning}. We will not touch upon the non-Markovianity issue, and focus on the  comparison between  PM and averaging. 
% {\color{blue} - I would shorten  this whole part in blue and move it somewhere else. \\
% In some cases, the distribution of the unresolved variables can be inferred (e.g.~from first principles), and it is commonly assumed that, in this case, the fidelity of a coarse-grained  model is higher when there is scale separation between the unresolved variables and the remaining ones {\color{red} and when the coupling between them is not too strong }\cite{givon2004extracting,gottwald2017stochastic}. A coarse-graining method that claims not to rely on such stringent assumptions is the Mori--Zwanzig projection operator formalism  \cite{grabert1982projection,zwanzig1973nonlinear}. It asserts to be an ``exact rewriting'' of the equations of motion that are then interpreted in the sense of a non-Markovian evolution equation for the variables of interest. Since the original equation is simply recast as another (more complicated) equation, it is not closed and therefore appropriate closures or approximations are necessary (e.g.~\cite{darve2009computing,hijon2010mori}). 

% It is often argued that the Mori--Zwanzig projection operator formalism is an extension of the optimal prediction projection operator formalism that can account for lack of scale separation \cite{chorin2002optimal}. Nevertheless, it is not clear, beyond special cases (e.g. \cite{kupferman2004fractional}), whether the high-fidelity Markovian equations can be obtained from the Mori--Zwanzig coarse-graining procedure when scale separation is present; cf.~\cite{zhang2016effective,lin2018milestoning}.}

\section{Some Background, Main Results and relation to literature}\label{sec: Main Results}

 We first give a minimal description of averaging, the PM and the Gy\"ongy method and introduce essential setup and notation - enough to be able to present the main results of the paper. More background on these methods and precise assumptions are deferred to  subsequent sections.

In what follows, we will consider systems that are either of the form 
\begin{subequations}\label{eq:sfep=1}
    \begin{alignat}{1}\label{eq:sfxep=1}
        dX_t & = f(X_t,Y_t)\,dt + \alpha(X_t,Y_t)\,dU_t\,, \qquad X_0=x\\ \label{eq:sfyep=1}
        dY_t & = g(X_t,Y_t)\,dt + \beta(X_t,Y_t)\,dW_t\,, \qquad Y_0=y\,,
    \end{alignat}
  \end{subequations}
  or of the form
  \begin{subequations}\label{eq:sf}
    \begin{alignat}{1}\label{eq:sfx}
        dX^\eps_t & = f(X^\eps_t,Y^\eps_t)\,dt + \alpha(X^\eps_t,Y^\eps_t)\,dU_t\,, \qquad X_0=x\\ \label{eq:sfy}
        dY^\eps_t & = \frac{1}{\eps}g(X^\eps_t,Y^\eps_t)\,dt + \frac{1}{\sqrt{\eps}}\beta(X_t\eep,Y_t\eep)\,dW_t\,, \qquad Y_0=y\,, 
    \end{alignat}
  \end{subequations}  
where $(X_t, Y_t), (X_t^{\epsilon}, Y_t^{\epsilon}) \in \R^{d}\times \R^{n-d}$,  $U$ and $W$ are independent Brownian motions in $\R^d$ and $\R^{n-d}$, respectively, and $f,\alpha, g, \beta$ are coefficients of appropriate dimensionality, namely $f\colon\R^d\times \R^{n-d} \rightarrow \R^d,  g\colon\R^d\times \R^{n-d} \rightarrow \R^{n-d}$ for the drift and $\alpha\colon\R^d\times \R^{n-d} \rightarrow \R^{d\times d}, \beta\colon\R^d\times \R^{n-d} \rightarrow \R^{(n-d)\times (n-d)}$ for the diffusion coefficients.
In the second system of equations,  $0<\eps<1$ is a parameter,   and sometimes  we will suppress the dependence on $\eps$ and simply write $X_t$ instead of $X^\eps_t$ to keep the notation compact. And moreover for any time-dependent process or function,  say $K_t$, we will use interchangeably the notation $K_t$ and $K(t)$. 

Unless otherwise stated, we will always assume that \eqref{eq:sfep=1} has a unique weak solution, that  the  coefficients of \eqref{eq:sfep=1} are smooth and that the system is uniformly elliptic. By the latter we mean that there exists $\zeta>0$ such that 
\begin{equation}\label{eqn:unif_ellipticity_initial_system}
    \langle \hat \Sigma(x,y) v, v \rangle \geq \zeta |v|^2, \quad \mbox{for every } (x,y) \in\R^{d}\times \R^{n-d}, v \in \R^n \,,   
\end{equation}
where $\hat\Sigma= \hat\sigma \hat\sigma^T$, with $\hat\sigma$ being the overall diffusion matrix of \eqref{eq:sfep=1},  and $\langle \cdot , \cdot \rangle, |\cdot | $ denote Euclidean scalar product and norm, respectively. 
Under these standing assumptions the  law of $Z_t:=(X_t, Y_t)$ has a strictly positive density on $\R^n$ for every $t>0$; we denote such a density by   $\rho_t=\rho_t(x,y)$.\footnote{Throughout this paper, unless otherwise stated, we assume that all the measures with which we deal have a density and, with abuse of notation,  we use the same letter for the measure and the density, so that e.g. we write $\nu(dx)=\nu(x) dx$ } 
We also assume that  $Z_t$  admits a unique invariant measure,  with density $\rho=\rho(x,y)$. The  $x-$marginals of $\rho_t$ and $\rho$ are denoted by $\bar \rho_t$ and $\bar \rho$, respectively: 
 $$
   \bar{\rho}_t(x):= \int_{\R^{n-d}} \rho_t (x,y)\, dy \, ,
  $$
  and 
  \begin{equation}\label{eqn:bar rho def}
  \bar{\rho}(x):= \int_{\R^{n-d}} \rho(x,y) \, dy \,.
  \end{equation}
  We will often need to condition with respect to the $x$- variables, so we introduce  
the conditional density $\rho_t(\cdot \vert x)$ at time $t\geq 0$ and the {\em equilibrium conditional density} (ECD) $\rho(\cdot \vert x)$    of system \eqref{eq:sfep=1} are then defined, namely:
\begin{equation}\label{joint and marginal time t}
\rho_t(y\vert x ) : = \frac{\rho_t(x,y)}{\int_{\R^{n-d}} \rho_t(x,y) \, dy}=\frac{\rho_t(x,y)}{\bar{\rho}_t(x)}\, , 
\end{equation}
and 
\begin{equation}\label{eqn:conditional for system 1}
     \rho(y\vert x) := \frac{\rho(x,y)}{\int_{\R^{n-d}} \rho(x,y) \, dy} = \frac{\rho(x,y)}{\bar\rho(x)}.
 \end{equation}

That is, if $(X,Y)$ is a random variable distributed according to $\rho$, then $\rho(\cdot \vert x)$ is the law of $Y$ given $X$; similarly for $\rho_t(\cdot \vert x)$. 

Let us now  introduce the Gy\"ongy method (GM), the PM and averaging, in turn. In all three cases we want to coarse grain system \eqref{eq:sfep=1} (or system \eqref{eq:sf}) and `eliminate' the variables $Y_t$ (or $Y_t\eep$), i.e. create a lower dimensional dynamics, of the same dimension as $X_t$. The resulting dynamics will be denoted by $X_G, X_P$ or $X_A$, depending on which method has been used.

\subsection{The Gy\"ongy method and the Projection method}
For the purposes of this paper, the main message of \cite{gyongy1986mimicking} can be summarised as follows.\footnote{The content of \cite{gyongy1986mimicking} is much richer, here we express it in a way that is convenient to us, but we will make more remarks on this later. } Suppose we have a system of the form \eqref{eq:sfep=1} which we want to coarse grain and in doing so we want the resulting  dynamics  $\{X_G(t)\}_{t\geq 0 } \subseteq \R^d$ to preserve the law of the $x$-marginal of the system; that is, we want $X_G(t)$ to be such that 
\begin{equation}\label{eqn:law equality gyongy}
Law(X_G(t)) = Law(X_t)  \quad \mbox{for every } t\geq 0.
\end{equation}
According to \cite{gyongy1986mimicking}, this is achieved by the following SDE
\begin{equation}\label{eqn:gyongy-coarse-graining}
    dX_G(t) = f_G(t,X_G(t)) dt+ \alpha_G^{1/2}(t, X_G(t)) dU_t
\end{equation}
where 
\begin{align}
    f_G(t,x) &:= \int_{\R^{n-d}} f(x,y) \rho_t(y\vert x) dy \, , \label{eqn:gyongy-averaged-coefficients} \\
    \alpha_G(t,x) &:= \int_{\R^{n-d}} \alpha\alpha^T(x,y) \rho_t(y\vert x) dy \,,\label{eqn:gyongy-averaged-coefficients2}
\end{align}
 where $\alpha_G^{1/2}$ denotes the positive definite square root of the positive definite matrix $\alpha_G$ and the superscript $^T$ denotes transpose.  
 We postpone a proof of the fact that \eqref{eqn:gyongy-coarse-graining} satisfies \eqref{eqn:law equality gyongy} to Subsection \ref{subsec:proof gyongy}.

 While the Gy\"ongy approach preserves the laws of the marginals for every $t\geq 0$, when applying the PM  the (minimal) requirement is that the resulting coarse grained dynamics $X_P(t)$ should have the same asymptotic behaviour as $X_t$. Assuming that $\rho_t$ converges to $\rho$ as $t\rightarrow \infty$ (in a suitable sense) then, under appropriate assumptions, $\bar \rho_t$ (which is by definition the law of $X_t$) converges to $\bar \rho$ as $t\rightarrow \infty$. Hence we require that the law of $X_P(t)$ should converge (weakly), as $t\rightarrow \infty$, to $\bar \rho$.

Coming back to the GM for a moment, let us  observe that,  since the law of $X_G(t)$ and the law of $X_t$ coincide for every $t\geq 0$, then  the asymptotic behaviour of these two processes  coincides as well. So, if we want to construct a process with the same asymptotic behaviour as $X_t$, we would in principle be done. However  explicitly simulating the process $X_G$ can be difficult, because it requires sampling  the conditional law $\rho_t(\cdot\vert x)$ at every time $t$. So in some sense the construction of the process $X_G$ has not achieved any real dimensionality reduction, as simulating/solving the equation for the process $X_G$ still potentially requires simulating the whole process \eqref{eq:sfep=1} in order to access the conditional distribution  $\rho_t(\cdot\vert x)$.  This is not surprising, as in going from $X_t$ to $X_G(t)$ no approximation has been made (the GM is an exact rewriting - this is explained in Section \ref{sec: sec 3}). 

The PM can then be viewed as an `approximation' of the Gy\"ongy  strategy in the sense that, since $\rho_t(\cdot\vert x)$ may be difficult or expensive to access, via the PM one  replaces projections on  $\rho_t(\cdot\vert x)$ with projections on the ECD $\rho(y\vert x)$. Namely, the dynamics  $X_P$ produced via the PM is the solution of the SDE
\begin{equation}\label{eq:sfgenaralprojection}
    dX_P(t) = f_P(X_P(t)) dt+ \alpha_P^{1/2}(X_P(t)) dU_t \,, \qquad X_P(0)=x \,, 
\end{equation}
where 
\begin{align}
    f_P(x) &:= \int_{\R^{n-d}} f(x,y) \rho(y\vert x) dy \, , \label{eqn:projection-coefficients} \\
    \alpha_P(x) &:= \int_{\R^{n-d}} \alpha \alpha^T(x,y) \rho(y\vert x) dy \,,\label{eqn:projection-coefficients2}
\end{align}
 and $\rho(\cdot \vert x)$ is the  ECD of \eqref{eq:sfep=1}, defined in \eqref{eqn:conditional for system 1}. 

Reference \cite{chorin2000optimal2} is probably one of the first works on the PM for Hamiltonian systems, but to our knowledge  the PM for SDEs was first introduced in \cite{legoll2010effective} and there applied to reversible systems of the form
\begin{subequations}\label{eq:sfgradient}
    \begin{alignat}{1}\label{eq:sfxgradient}
        dx_t & = -\nabla_xV(x_t,y_t)\,dt + \sqrt{2}\,dU_t \, , \qquad x_0=x\\ \label{eq:sfygradient}
        dy_t & = -\nabla_yV(x_t,y_t)\,dt + \sqrt{2}\,dW_t\,, \qquad y_0=y \,,
    \end{alignat}
  \end{subequations}
  where $V\colon\R^n \rightarrow \R $ is a smooth confining potential. \footnote{By {\em confining} we mean that $V(z) \rightarrow \infty $ when $|z|\rightarrow \infty$. } Of course \eqref{eq:sfgradient} is a special case of \eqref{eq:sfep=1}. 
  The invariant measure $\rho= \rho_V$ of \eqref{eq:sfgradient} is explicitly known, 
  \begin{equation}\label{eqn:gradient inv meas}
  \rho_V(x,y) = e^{-V(x,y)}
  \end{equation}
  (assuming here for simplicity that the normalization constant for $\rho_V$ is contained in $V$), so that in this case also  the corresponding  marginal $\bar\rho_V$ and conditional $\rho_V(\cdot \vert x)$ are known
 
  \begin{equation}\label{cond meas gradient case}
 \bar \rho_V(x) := \int_{\R^{n-d}} e^{-V(x,y)} dy \, , \qquad \rho_V(y\vert x) := \frac{\rho_V(x,y)}{\bar{\rho}_V(x)} \stackrel{\eqref{eqn:gradient inv meas}}{=} \frac{e^{-V(x,y)}}{\int_{\R^{n-d}}e^{-V(x,y)}dy} \,.
\end{equation}
  Applying the PM to \eqref{eq:sfgradient} produces the dynamics 
\begin{equation}\label{eq:projection-gradientsystem}
dx_P(t) = -v(x_P(t)) dt + \sqrt{2} dU_t \, , \qquad x_P(0)=x \,,
\end{equation}
with
$$
v(x):= \int_{\R^{n-d}} \nabla_x V(x,y) \rho_V(y\vert x) dy \,.
$$
It is easy to see that $\bar{\rho}_V$ is an invariant measure for  \eqref{eq:projection-gradientsystem}.  Indeed the following  simple calculation shows that \eqref{eq:projection-gradientsystem} is precisely a Langevin dynamics with drift given by $\nabla_x \log \bar{\rho}_V(x)$, and hence (under well-studied assumptions on $V$, \cite{pavliotis2014stochastic}) it will sample from $\bar{\rho}_V$: 
\begin{align} 
    \nabla_x \log \bar{\rho}_V(x) &= \nabla_x \log \left(\int e^{-V(x,y)} \, dy \right)
     = -\frac{\int \nabla_xV(x,y) e^{-V(x,y)} \, dy}{\int e^{-V(x,y)} \, dy} \nonumber\\
     &= 
    - \int \nabla_xV(x,y) \rho_V(y\vert x)  \, dy = -v(x)\,.\label{proj method gradient explicit caculation}
\end{align}
In the physics literature, starting from Kirkwood \cite{kirkwood1935statistical}, the negative grad log marginal density, $-\log\bar{\rho}_V$, drift is sometimes called the potential of mean force or free energy, owing to the fact that the projected drift $v$ is a gradient force.
The explicit knowledge of the measure $\rho_V$, and hence of $\bar\rho_V$ and $\rho_V(\cdot \vert x)$, can also be exploited to prove that the law of $x_P(t)$ converges to $\bar\rho_V$ as $t\rightarrow \infty$ \cite{legoll2010effective}.

If system \eqref{eq:sfep=1} is not in gradient form or, more in general, if the explicit form of the invariant measure $\rho$ is not known then %we cannot carry out a calculation as explicit as the one in \eqref{proj method gradient explicit caculation} but 
it is still easy to see that $\bar \rho$ is invariant for $X_P$, see \cite{zhang2016effective} and Proposition \ref{lem:bar rho invariant for Lp}. However proving convergence of $X_P$ to the equilibrium marginal $\bar \rho$ is more challenging. This is due to the fact that proving convergence to equilibrium for systems where the invariant measure is not a priori known is, notoriously, a difficult problem. 

To circumvent this issue we first observe that, for the purposes of our analysis,  the dynamics \eqref{eqn:gyongy-coarse-graining} obtained via the GM can be regarded as  non-autonomous (more comments on this in Section \ref{sec: sec 3}) and then use the relationship between the PM and the Gy\"ongy approximation and study the  long time behaviour of the process $X_P$ in \eqref{eq:sfgenaralprojection}   by combining the results of \cite{gyongy1986mimicking} (which we slightly adapt to our context in Proposition \ref{Prop: extension gyongy})  with 
 some intuitive (but not simple) results of \cite{angiuli2013hypercontractivity, cass2021long}, on the long time behaviour of non-autonomous SDEs.  

%This is what we prove in Theorem \ref{thm:projection method works}, which is one of the main results of this paper. 

 One of the main results of this paper is then Theorem \ref{thm:projection method works}, which we state and prove in Section \ref{sec: statement and proof of Thm 6.2}. Informally, the theorem can be stated as follows. 
 
\smallskip\noindent
 {\bf Informal statement of Theorem \ref{thm:projection method works}.} {\em Under appropriate  assumptions (which guarantee that both $X_G$ and $X_P$ are well defined and admit a unique invariant measure), the process $X_G$ and the process $X_P$ have the same long-time behaviour. Recalling that,  
  as a consequence of property \eqref{eqn:law equality gyongy}, $X_G$ converges weakly to $\bar\rho$ as $t\rightarrow \infty$, this implies that  also $X_P$ converges to $\bar \rho$ as $t\rightarrow \infty$.}
  
  \smallskip\noindent
Theorem \ref{thm:projection method works} then serves two purposes: on the one hand it makes the relation between the Gy\"ongy approximation $X_G$ and the PM approximation $X_P$ rigorous; on the other hand it gives sufficient conditions under which the dynamics produced via the PM has the desired asymptotic  property.
 We emphasize that our proof does not make use of any specific structure of system \eqref{eq:sfep=1}.  In particular, while previous analysis of the PM was either tailored to reversible diffusions \cite{legoll2010effective,zhang2016effective, legoll2017pathwise,lelievre2019pathwise,nuske2021spectral} or to  specific non-reversible diffusions for which the form of the invariant measure is still known \cite{duong2018quantification, legoll2019effective,hartmann2020coarse}, our analysis bypasses the issue of reversibility entirely. Nonetheless, the results of this paper are only concerned with weak convergence to $\bar\rho$, and we do not study pathwise properties. 

 Throughout the paper we comment on the extent to which the conditions of Theorem \ref{thm:projection method works} are sharp and  below we give an example of a process for which both the GM and the PM fail, see Example \ref{example:short version}.  In Section \ref{sec: sec 3} we also present the heuristic behind the proof of Theorem \ref{thm:projection method works}. Technically speaking, the difficulty in proving Theorem \ref{thm:projection method works} comes from the fact that the long-time behaviour of  non-autonomous dynamics is typically described by objects more complicated than a single invariant measure, see Section \ref{sec: statement and proof of Thm 6.2} for more background on this.  It is the specific structure of the coefficients of \eqref{eqn:gyongy-coarse-graining} that allows to conclude that the asymptotic behaviour of  \eqref{eqn:gyongy-coarse-graining} is indeed described by a unique invariant measure, which turns out to be the same invariant measure of the PM dynamics $X_P$.

%%%%%%%%%%%%%%%%%%%%%%%%%
%%%%%%%%%%%%%%%%
%%%%%%%%%%%Qui Qui Qui

Let us now come to the second main contribution of this paper; to explain it we first need to introduce averaging. 

\subsection{Averaging for SDEs. }
We have presented the PM and GM for systems of the form \eqref{eq:sfep=1}, but both can also be applied to models of the form \eqref{eq:sf}. We will come back to this later. Averaging, on the other hand,  is more commonly applied to  systems of the form \eqref{eq:sf}, so-called {\em slow-fast} systems. Indeed, due to the scaling properties of Brownian motion,  as $\eps$ tends to zero $Y_t\eep$ moves faster and faster, so that the parameter $\eps$ defines (in an explicit way) the scale separation between the slow process (SP) $X_t\eep$ and the fast process $Y_t\eep$. The behaviour as $\eps \rightarrow 0$ of slow-fast processes   is studied via {\em (stochastic) averaging}, the basics of which can be summarised as follows.   Since $Y_t\eep$ evolves much faster than $X_t\eep$, intuitively, the former process will have `reached equilibrium' while the latter has remained substantially unchanged. Hence, to obtain a coarse grained description of \eqref{eq:sf}, one fixes the value of the slow process, say to $X_t\eep=x$, and considers an intermediate process, the so-called `Frozen Process' or `auxiliary process', namely the process $Y^{(x)}_t$ which solves the following SDE
\begin{equation}\label{eq:fastSDE}
            dY^{(x)}_t = g(x,Y^{(x)}_t)\,dt + \beta(x,Y^{(x)}_t)\,dW_t\, \, \quad Y_0^{(x)}=y \,. 
\end{equation}
We emphasize that in the above $x$ is fixed but arbitrary, so $\{Y^{(x)}_t\}_{x \in \R^d}$ 
can be seen as a one-parameter family of processes, parametrized by $x$.
 Assuming that $Y_t^{(x)}$ is ergodic, i.e. that it admits, for each $x$ fixed, a unique invariant measure, $\rho^{(x)}$,  and moreover that $Y_t^{(x)}$ converges, in a suitable sense,  to $\rho^{(x)}$ (irrespective of the initial state $y$)  one can show that, as $\eps$ tends to zero, the slow process $X_t\eep$ converges (at least weakly) to the so called averaged dynamics, $X_A(t)$, namely to the solution of the SDE whose coefficients are obtained from those of the process  \eqref{eq:sfx} by `averaging them' with respect to (or `projecting them' on) the   measure $\rho^{(x)}$ (which, for each $x$ fixed, is a measure on $\R^{n-d}$):
\begin{equation}\label{eqn:averaged dynamics intro}
 d X_A(t)=   f_A (X_A(t)) dt + \sigma_A^{1/2}(X_A(t)) dU_t\, ,  
\end{equation}
where
\begin{equation}\label{eqn:averaged coefficients}
 f_A(x) = \int_{\R^{n-d}} f(x,y) \rho^{(x)}(y) dy \,, \quad  \sigma_A(x) = \int_{\R^{n-d}} (\alpha\alpha^T)(x,y) \rho^{(x)}(y) dy \,. 
\end{equation}
 Notice that the resulting dynamics $X_A(t)$ is a dynamics of the same dimension as $X\eep(t)$, so averaging has achieved dimension reduction by exploiting time-scale separation. The averaged dynamics $X_A(t)$ \eqref{eqn:averaged dynamics intro} typically approximates the process  \eqref{eq:sfx}  on finite time intervals and, under conditions on the coefficients,   on infinite-time horizons as well \cite{pavliotis2008multiscale, crisan2022poisson, schuh2024conditions}.    That is, one can typically provide quantitative estimates of the form 
 \begin{equation}\label{bound:Uit averaging}
    \left \vert \mathbb E u (X\eep(t)) - \mathbb E u(X_A(t))
    \right \vert \leq C \varepsilon^{\gamma}
 \end{equation}
 for some $\gamma>0$, functions $u: \R^{d} \rightarrow \R$ in an appropriate class,   and $C$ a constant that may or may not depend on time. When $C$ does not depend on time,  we say that $X\eep$ converges to $X_A$ uniformly in time. 

\subsection{Comparing averaging and the projection method}\label{subsec:comparison}
Next, we address the following question

\begin{equation}
  \tag{Q}\label{Q}
  \parbox{\dimexpr\linewidth-4em}{%
    \strut
    \emph{In presence  of  scale  separation, do  projection  method and averaging  coincide?}
    \strut
  }
\end{equation}

\smallskip
\noindent
That is, is it the case that the dynamics $X_P$ in \eqref{eq:sfgenaralprojection} and the dynamics $X_A$ in \eqref{eqn:averaged dynamics intro} coincide, when there is some scale separation in the system?. In this paper we phrase this question precisely and provide sufficient conditions under which the answer is affirmative.  More importantly,  we point out that  is in general the answer needs not be affirmative, as the following example shows. 

%%%%%%%%%%%%%
%%%%%%%%%%%%%%%%%
%%%%%%%%%%%%%%%%%%%%  HERE HERE HERE

\begin{Example}\label{example:short version}\textup{
Consider the following  two-dimensional Ornstein-Uhlenbeck process:
    \begin{subequations}\label{eq:sfeps}
\begin{align}
        dX^\eps_t &= -\left(X^\eps_t - Y^\eps_t\right) dt \,,\quad X^\eps_0=x \in \R\label{example:Carsten'sexample1eps}\\
dY^\eps_t &= -\frac{1}{\eps}Y^\eps_t \,  dt + \sqrt{\frac{2}{\eps}} \, dW_t \,,\quad Y^\eps_0=y 
 \in \R \label{example:Carsten'sexample2eps} \,.
    \end{align}
  \end{subequations}
  To calculate the averaging approximation we observe that the associated frozen process is given by 
 $$
 dY_t^{(x)} = - Y_t^{(x)} dt + \sqrt{2}  dW_t,  
 $$
 with invariant measure $\rho^{(x)} \sim \cN(0,1)$ (here neither the frozen process nor its stationary measure $\rho^{(x)}$ depend on $x$, as the second equation is decoupled from the first one, we keep $x$ in the notation merely for consistency).\footnote{An explicit calculation using variation of constants shows that $Y_t$ is Gaussian $m_t=e^{-t}y$ and variance $\sigma^2(t)=1-e^{-2t}$}
 As a consequence, according to \eqref{eqn:averaged dynamics intro},  coarse graining via averaging produces a process $\xa$ which solves the differential equation 
\begin{equation}\label{example:Carsten'sexampleAvg}
    d \xa(t) = - \xa(t)\, dt,\quad \xa(0)=x\,.
\end{equation}
To apply the PM we need to calculate the invariant measure  of system \eqref{example:Carsten'sexample1eps}-\eqref{example:Carsten'sexample2eps} and then the associated ECD . Note that such  measures will in general depend on $\eps$, so we denote them $\rho\eep$ and $\rho\eep( \cdot \vert x)$, respectively.  By direct calculation one can see that $\rho\eep$ is the centered Gaussian    
\begin{equation}
\label{Gaussian invariant}
\rho\eep=\cN(0,\Sigma\eep)\,,\quad \Sigma=\begin{pmatrix}
    \frac{\eps}{1+\eps} & \frac{\eps}{1+\eps} \\ \frac{\eps}{1+\eps} & 1
\end{pmatrix}\,,    
\end{equation}
 and (by completing the square in the PDF associated with $\rho\eep$) the ECD is 
\begin{equation}
\label{eq: ECD}
\rho\eep(\cdot|x) = \cN\left(x,1-\frac{\eps}{1+\eps}\right)\,.
\end{equation}
Hence, using \eqref{eq:sfgenaralprojection}, via the PM we  obtain a coarse grained  equation that is different from (\ref{example:Carsten'sexampleAvg}): 
\begin{equation}\label{example:Carsten'sexamplePM}
d\xp(t) = 0\,,\quad \xp(0)=x\,.
\end{equation}
Note that in this example the equation obtained by the PM does not depend on $\eps$, but this needs not be the case.
We will come back to this example in much more detail in Subsection \ref{sec: counterexample to Gyongy}, but we anticipate that even  in this relatively simple case not only averaging and PM do not give the same result; it is also the case that the dynamics \eqref{example:Carsten'sexamplePM} does not capture the long-time behaviour of \eqref{example:Carsten'sexample1eps}, so the PM fails here. Furthermore, as we will show in Subsection \ref{sec: counterexample to Gyongy}, for this example the coarse grained dynamics obtained via GM has some non-trivial features. 
\hfill{$\Box$}}
\end{Example}

Example \ref{example:short version} is a rather natural example, in the sense that, while being simple and lower dimensional, it still retains all the important features of some of the most popular models used e.g. in molecular dynamics, see e.g. Markovian approximations of the generalised Langevin equation \cite{kupferman2004fractional,ottobre2011asymptotic}. 

A closer inspection of this example also helps to show that, when trying to answer question \eqref{Q} a few issues compound. We spell out each of them here, relating them to the main results of this paper.

In system \eqref{eq:sf} there is an {\em explicit } scale separation. When considering system \eqref{eq:sfep=1} in the context of various projection-type methods it is often assumed that there is an {\em implicit} time-scale separation in the dynamics. To explain what is commonly meant  by this, suppose $\alpha=\beta=0$ and that 
$$
f(x,y)= Ax+ h(x,y), \quad g(x,y)= By + \tilde{h}(x,y) \, ,
$$
where $A,B$ are matrices with real eigenvalues and $h,\tilde{h}$ are perturbations,  in the sense  that the behaviour of the dynamics is governed by the linear parts described by $A$ and $B$, \cite[Section 4]{givon2004extracting}.  If the largest  eigenvalue of $B$, $\lambda_B$,  is e.g. negative and  (much) smaller than the smallest eigenvalue of $A$, $\lambda_A$,  then there is still a scale separation in the system - this is what we call an implicit scale separation. Ways of quantifying this implicit scale separation for more general systems have been suggested e.g. in \cite{legoll2010effective}, via log-Sobolev inequalities.  Note that if in \eqref{eq:sf} $f,g, \alpha, \beta$ are chosen as in the above then the eigenvalue dominating the dynamics of $Y_t\eep$ is $\lambda_B/\eps$, which is arbitrarily smaller than $\lambda_A$, for $\eps$ small enough. This is another way of seeing that the parameter $\eps$ governs the explicit scale separation in \eqref{eq:sf}. Whether the time-scale separation is explicit or implicit, the averaging approximation  still holds (it is just the case that the parameter $\eps$ allows one to quantify  the approximation  error as a function of the scale-separation). Hence, in systems with whichever form of scale separation, both  PM and averaging can be applied, so that it makes sense to compare them.  We compare such methods both in  the case in which the dynamics to which they are applied is   of the form \eqref{eq:sfep=1} with implicit scale separation and in the setting of models of the form \eqref{eq:sf}. 

\subsubsection*{Implicit scale separation}
%$\bullet$ {\em Implicit scale separation}.   
To motivate question \eqref{Q}  we observe that, when system \eqref{eq:sfep=1} is reversible, i.e.~of the form \eqref{eq:sfgradient},  then  the family of equilibrium measures  $\{\rho^{(x)}\}_{x \in \R^d}$ of the fast process 
 coincides with the conditional distributions $\{\rho(\cdot \vert x)\}_{x\in\R^d}$. That is, 
 \begin{equation}\label{eqn:frozen=conditional}
    \rho^{(x)}(y) = \rho(y\vert x), \quad \mbox{for every } x \in \R^d. 
\end{equation} 
Indeed the auxiliary process associated to \eqref{eq:sfgradient} is given by 
$$
dy_t^{(x)}  = -\nabla_y V(x,y_t^{(x)})\,dt + \sqrt{2}\,dW_t\, .
$$
For each $x \in \R^d$, the above process is a Langevin diffusion in the $y$ variable; hence its equilibrium measure is 
$$
\rho^{(x)}(y) = \frac{e^{-V(x,y)}}{\int e^{-V(x,y)} \, dy }\,,
$$
 which coincides with the expression \eqref{cond meas gradient case} for $\rho_V(\cdot \vert x)$.   Hence in the gradient case, the coarse grained dynamics obtained via the PM and the one obtained via averaging coincide too. Setting $\varepsilon=1$ (or to any other number indeed) in  Example \ref{example:short version}, one can see that equality \eqref{eqn:frozen=conditional} does not hold in that case. This is consistent with the fact that our little example is non-reversible. One might therefore wonder whether equality \eqref{eqn:frozen=conditional} fails to hold because of lack of reversibility. In Section \ref{sec: rhox and rhox|y} we will show through examples and counterexamples that the lack of reversibility plays no role. We will in fact construct two non-reversible systems; one of them is such that $\rho^{(x)} \neq \rho(\cdot \vert x)$ while for the other these two families of measures coincide, see Example \ref{example: rhox not equal rhoygivenx} and Example \ref{example: rhox=rhoygivenx}.  In Section \ref{sec: rhox and rhox|y} we will show that  whether these two measures, and hence the respective associated approximations, coincide or not, depends crucially on the structure of  the kernels of the operators $\cL_x'$ and $\cL_y'$, duals of the generators associated to the slow and fast dynamics, respectively;  see \eqref{flat adjoint} and Proposition \ref{prop: equivalent conditions}.

From a broader perspective, we point out that it is not surprising that \eqref{eqn:frozen=conditional} does not always hold. Indeed, loosely speaking, the measure $\rho(\cdot \vert x)$ is obtained by first letting $t\rightarrow \infty$ in system \eqref{eq:sf} and then conditioning on the $x$-variable; while $\rho^{(x)}$ is obtained by first conditioning the dynamics for $Y_t$ with respect to $x$ and then letting $t\rightarrow \infty$. And of course the commutativity of the two operations is not guaranteed.

\subsubsection*{Explicit scale separation}
%$\bullet ${\em Explicit scale separation. }  
In order to apply the PM to \eqref{eq:sf} we need to calculate the associated ECD.   In this case the ECD may in general  depend on $\eps$, $\rho\eep(\cdot \vert x)$ while the measure $\rho^{(x)}$ never depends on $\eps$, by construction. So, in this case, a better way of phrasing question \eqref{Q} is to ask whether the measures $\rho\eep(\cdot \vert x)$ and $\rho^{(x)}$  coincide in the limit of scale separations, i.e.  whether the following  limit holds,  
\begin{equation}\label{limit scale separation}
    \lim_{\eps \rightarrow 0} \rho^{\eps}(y\vert x) = \rho^{(x)}(y) \, ,
\end{equation}
in some appropriate sense. 
Again, Example \ref{example:short version} shows that this needs not be the case.  
In Section \ref{sec: Carsten's section} we look at this limit through a formal asymptotic expansion in $\eps$ and  give sufficient conditions under which this limit holds. 

Finally, in the presence of scale separation one is led to study the commutativity of the diagram in Figure \ref{fig:diagram}. In this paper we do not address the full study of the conditions under which this diagram is fully commutative. However we make a few remarks on the diagram (which relate to the main results of the paper).

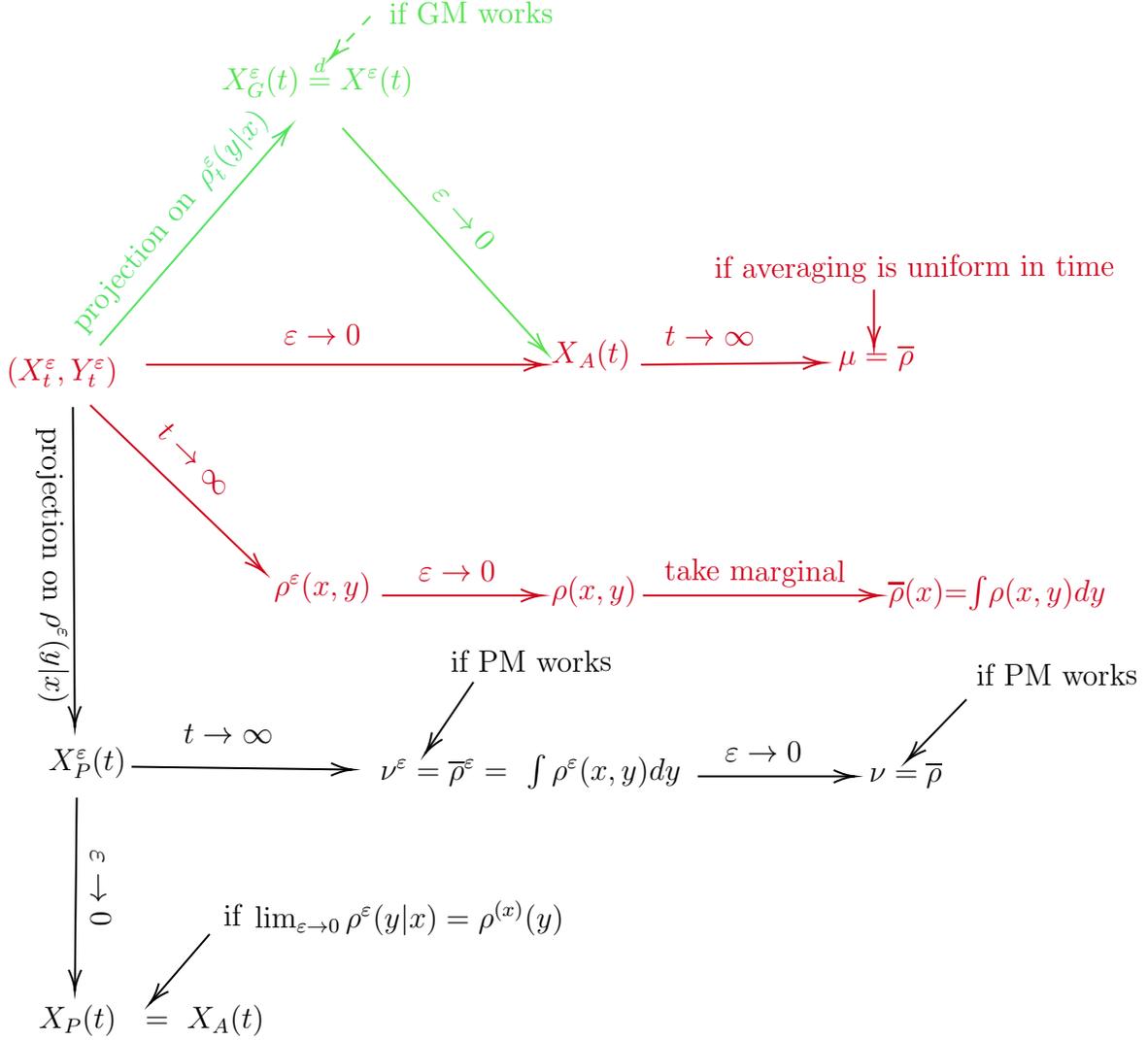
\begin{figure}
    \centering
    
\tikzset{every picture/.style={line width=0.75pt}} %set default line width to 0.75pt        

\begin{tikzpicture}[x=0.75pt,y=0.75pt,yscale=-1,xscale=1]
%uncomment if require: \path (0,553); %set diagram left start at 0, and has height of 553

%Straight Lines [id:da332478906096934] 
\draw [color={rgb, 255:red, 208; green, 2; blue, 27 }  ,draw opacity=1 ]   (118,191) -- (323,191) ;
\draw [shift={(325,191)}, rotate = 180] [color={rgb, 255:red, 208; green, 2; blue, 27 }  ,draw opacity=1 ][line width=0.75]    (10.93,-3.29) .. controls (6.95,-1.4) and (3.31,-0.3) .. (0,0) .. controls (3.31,0.3) and (6.95,1.4) .. (10.93,3.29)   ;
%Straight Lines [id:da8910024181976837] 
\draw [color={rgb, 255:red, 208; green, 2; blue, 27 }  ,draw opacity=1 ]   (375,191) -- (468,190.02) ;
\draw [shift={(470,190)}, rotate = 179.4] [color={rgb, 255:red, 208; green, 2; blue, 27 }  ,draw opacity=1 ][line width=0.75]    (10.93,-3.29) .. controls (6.95,-1.4) and (3.31,-0.3) .. (0,0) .. controls (3.31,0.3) and (6.95,1.4) .. (10.93,3.29)   ;
%Straight Lines [id:da8313801784015336] 
\draw [color={rgb, 255:red, 208; green, 2; blue, 27 }  ,draw opacity=1 ]   (89,212) -- (178.06,297.61) ;
\draw [shift={(179.5,299)}, rotate = 223.87] [color={rgb, 255:red, 208; green, 2; blue, 27 }  ,draw opacity=1 ][line width=0.75]    (10.93,-3.29) .. controls (6.95,-1.4) and (3.31,-0.3) .. (0,0) .. controls (3.31,0.3) and (6.95,1.4) .. (10.93,3.29)   ;
%Straight Lines [id:da054778742352675325] 
\draw [color={rgb, 255:red, 208; green, 2; blue, 27 }  ,draw opacity=1 ]   (240.5,311) -- (321.5,311) ;
\draw [shift={(323.5,311)}, rotate = 180] [color={rgb, 255:red, 208; green, 2; blue, 27 }  ,draw opacity=1 ][line width=0.75]    (10.93,-3.29) .. controls (6.95,-1.4) and (3.31,-0.3) .. (0,0) .. controls (3.31,0.3) and (6.95,1.4) .. (10.93,3.29)   ;
%Straight Lines [id:da7444990747399827] 
\draw    (80,213) -- (80.99,377.5) ;
\draw [shift={(81,379.5)}, rotate = 269.66] [color={rgb, 255:red, 0; green, 0; blue, 0 }  ][line width=0.75]    (10.93,-3.29) .. controls (6.95,-1.4) and (3.31,-0.3) .. (0,0) .. controls (3.31,0.3) and (6.95,1.4) .. (10.93,3.29)   ;
%Straight Lines [id:da6976058693802791] 
\draw [color={rgb, 255:red, 0; green, 0; blue, 0 }  ,draw opacity=1 ]   (110.5,400) -- (223,401.47) ;
\draw [shift={(225,401.5)}, rotate = 180.75] [color={rgb, 255:red, 0; green, 0; blue, 0 }  ,draw opacity=1 ][line width=0.75]    (10.93,-3.29) .. controls (6.95,-1.4) and (3.31,-0.3) .. (0,0) .. controls (3.31,0.3) and (6.95,1.4) .. (10.93,3.29)   ;
%Straight Lines [id:da0006226223880040349] 
\draw    (82,416) -- (81.02,514.5) ;
\draw [shift={(81,516.5)}, rotate = 270.57] [color={rgb, 255:red, 0; green, 0; blue, 0 }  ][line width=0.75]    (10.93,-3.29) .. controls (6.95,-1.4) and (3.31,-0.3) .. (0,0) .. controls (3.31,0.3) and (6.95,1.4) .. (10.93,3.29)   ;
%Straight Lines [id:da19703024724041085] 
\draw [color={rgb, 255:red, 80; green, 227; blue, 94 }  ,draw opacity=1 ]   (94,181) -- (191.69,68.51) ;
\draw [shift={(193,67)}, rotate = 130.97] [color={rgb, 255:red, 80; green, 227; blue, 94 }  ,draw opacity=1 ][line width=0.75]    (10.93,-3.29) .. controls (6.95,-1.4) and (3.31,-0.3) .. (0,0) .. controls (3.31,0.3) and (6.95,1.4) .. (10.93,3.29)   ;
%Straight Lines [id:da7203281352285095] 
\draw [color={rgb, 255:red, 87; green, 227; blue, 80 }  ,draw opacity=1 ]   (220,68) -- (324.66,183.52) ;
\draw [shift={(326,185)}, rotate = 227.82] [color={rgb, 255:red, 87; green, 227; blue, 80 }  ,draw opacity=1 ][line width=0.75]    (10.93,-3.29) .. controls (6.95,-1.4) and (3.31,-0.3) .. (0,0) .. controls (3.31,0.3) and (6.95,1.4) .. (10.93,3.29)   ;
%Straight Lines [id:da7960967201585865] 
\draw [color={rgb, 255:red, 87; green, 227; blue, 80 }  ,draw opacity=1 ] [dash pattern={on 4.5pt off 4.5pt}]  (234.5,10) -- (214.88,30.55) ;
\draw [shift={(213.5,32)}, rotate = 313.67] [color={rgb, 255:red, 87; green, 227; blue, 80 }  ,draw opacity=1 ][line width=0.75]    (10.93,-3.29) .. controls (6.95,-1.4) and (3.31,-0.3) .. (0,0) .. controls (3.31,0.3) and (6.95,1.4) .. (10.93,3.29)   ;
%Straight Lines [id:da7222023819072494] 
\draw [color={rgb, 255:red, 208; green, 2; blue, 27 }  ,draw opacity=1 ]   (376,311) -- (497,311) ;
\draw [shift={(499,311)}, rotate = 180] [color={rgb, 255:red, 208; green, 2; blue, 27 }  ,draw opacity=1 ][line width=0.75]    (10.93,-3.29) .. controls (6.95,-1.4) and (3.31,-0.3) .. (0,0) .. controls (3.31,0.3) and (6.95,1.4) .. (10.93,3.29)   ;
%Straight Lines [id:da03161806910668885] 
\draw [color={rgb, 255:red, 208; green, 2; blue, 27 }  ,draw opacity=1 ]   (496,152) -- (496,180) ;
\draw [shift={(496,182)}, rotate = 270] [color={rgb, 255:red, 208; green, 2; blue, 27 }  ,draw opacity=1 ][line width=0.75]    (10.93,-3.29) .. controls (6.95,-1.4) and (3.31,-0.3) .. (0,0) .. controls (3.31,0.3) and (6.95,1.4) .. (10.93,3.29)   ;
%Straight Lines [id:da6709459344489559] 
\draw    (288,356) -- (264.12,391.34) ;
\draw [shift={(263,393)}, rotate = 304.05] [color={rgb, 255:red, 0; green, 0; blue, 0 }  ][line width=0.75]    (10.93,-3.29) .. controls (6.95,-1.4) and (3.31,-0.3) .. (0,0) .. controls (3.31,0.3) and (6.95,1.4) .. (10.93,3.29)   ;
%Straight Lines [id:da496754253763117] 
\draw [color={rgb, 255:red, 0; green, 0; blue, 0 }  ,draw opacity=1 ]   (404.5,405) -- (484,405) ;
\draw [shift={(486,405)}, rotate = 180] [color={rgb, 255:red, 0; green, 0; blue, 0 }  ,draw opacity=1 ][line width=0.75]    (10.93,-3.29) .. controls (6.95,-1.4) and (3.31,-0.3) .. (0,0) .. controls (3.31,0.3) and (6.95,1.4) .. (10.93,3.29)   ;
%Straight Lines [id:da5680262804339274] 
\draw    (544,365) -- (514.39,395.56) ;
\draw [shift={(513,397)}, rotate = 314.09] [color={rgb, 255:red, 0; green, 0; blue, 0 }  ][line width=0.75]    (10.93,-3.29) .. controls (6.95,-1.4) and (3.31,-0.3) .. (0,0) .. controls (3.31,0.3) and (6.95,1.4) .. (10.93,3.29)   ;
%Straight Lines [id:da18225982318266665] 
\draw    (151,487) -- (121.31,521.48) ;
\draw [shift={(120,523)}, rotate = 310.73] [color={rgb, 255:red, 0; green, 0; blue, 0 }  ][line width=0.75]    (10.93,-3.29) .. controls (6.95,-1.4) and (3.31,-0.3) .. (0,0) .. controls (3.31,0.3) and (6.95,1.4) .. (10.93,3.29)   ;

% Text Node
\draw (44,185.4) node [anchor=north west, inner sep=0.75pt] 
    {$\textcolor[rgb]{0.82,0.01,0.11}{\left( X_t^{\varepsilon}, Y_t^{\varepsilon} \right)}$};

%\draw (44,185.4) node [anchor=north west][inner sep=0.75pt]    {$\textcolor[rgb]{0.82,0.01,0.11}{\left(\mathnormal{\textcolor[rgb]{0.82,0.01,0.11}{X}\textcolor[rgb]{0.82,0.01,0.11}{_{t}^{\varepsilon }}\textcolor[rgb]{0.82,0.01,0.11}{,Y}}\textcolor[rgb]{0.82,0.01,0.11}{_{t}^{\varepsilon }}\textcolor[rgb]{0.82,0.01,0.11}{\right)}$};
% Text Node
\draw (188,168.4) node [anchor=north west][inner sep=0.75pt]  [color={rgb, 255:red, 208; green, 2; blue, 27 }  ,opacity=1 ]  {$\varepsilon \rightarrow 0$};
% Text Node
\draw (327,176.4) node [anchor=north west][inner sep=0.75pt]  [color={rgb, 255:red, 208; green, 2; blue, 27 }  ,opacity=1 ]  {$\textcolor[rgb]{0.82,0.01,0.11}{X}\textcolor[rgb]{0.82,0.01,0.11}{_{A}}\textcolor[rgb]{0.82,0.01,0.11}{(}\textcolor[rgb]{0.82,0.01,0.11}{t}\textcolor[rgb]{0.82,0.01,0.11}{)}$};
% Text Node
\draw (386,169.4) node [anchor=north west][inner sep=0.75pt]  [color={rgb, 255:red, 208; green, 2; blue, 27 }  ,opacity=1 ]  {$t\rightarrow \infty $};
% Text Node
\draw (476,179.4) node [anchor=north west][inner sep=0.75pt]  [color={rgb, 255:red, 74; green, 144; blue, 226 }  ,opacity=1 ]  {$\textcolor[rgb]{0.82,0.01,0.11}{\mu =\ }\textcolor[rgb]{0.82,0.01,0.11}{\overline{\rho }}$};
% Text Node
\draw (129.31,217.39) node [anchor=north west][inner sep=0.75pt]  [color={rgb, 255:red, 208; green, 2; blue, 27 }  ,opacity=1 ,rotate=-44.47,xslant=-0.14]  {$t\rightarrow \infty $};
% Text Node
\draw (183,297.4) node [anchor=north west][inner sep=0.75pt]  [color={rgb, 255:red, 208; green, 2; blue, 27 }  ,opacity=1 ]  {$\rho ^{\varepsilon }( x,y)$};
% Text Node
\draw (258,291.4) node [anchor=north west][inner sep=0.75pt]  [color={rgb, 255:red, 208; green, 2; blue, 27 }  ,opacity=1 ]  {$\varepsilon \rightarrow 0$};
% Text Node
\draw (327,299.4) node [anchor=north west][inner sep=0.75pt]  [color={rgb, 255:red, 208; green, 2; blue, 27 }  ,opacity=1 ]  {$\textcolor[rgb]{0.82,0.01,0.11}{\rho }\textcolor[rgb]{0.82,0.01,0.11}{(}\textcolor[rgb]{0.82,0.01,0.11}{x,y}\textcolor[rgb]{0.82,0.01,0.11}{)}$};
% Text Node
\draw (66,388.4) node [anchor=north west][inner sep=0.75pt]    {$X_{P}^{\varepsilon }( t)$};
% Text Node
\draw (79.89,318.09) node [anchor=north west][inner sep=0.75pt]  [rotate=-90.63]  {$\rho ^{\varepsilon }( y|x)$};
% Text Node
\draw (76.06,221.78) node [anchor=north west][inner sep=0.75pt]  [rotate=-89.52] [align=left] {projection on};
% Text Node
\draw (238,392.4) node [anchor=north west][inner sep=0.75pt]    {$\nu ^{\varepsilon } =\overline{\rho }^{\varepsilon } =\ \int \rho ^{\varepsilon }( x,y) dy$};
% Text Node
\draw (100.94,441.44) node [anchor=north west][inner sep=0.75pt]  [color={rgb, 255:red, 0; green, 0; blue, 0 }  ,opacity=1 ,rotate=-89.58]  {$\varepsilon \rightarrow 0$};
% Text Node
\draw (61,522.4) node [anchor=north west][inner sep=0.75pt]    {$X_{P}( t) \ \ =\ X_{A}( t) \ \ \ \ $};
% Text Node
\draw (136,376.4) node [anchor=north west][inner sep=0.75pt]  [color={rgb, 255:red, 0; green, 0; blue, 0 }  ,opacity=1 ]  {$t\rightarrow \infty $};
% Text Node
\draw (156,28.4) node [anchor=north west][inner sep=0.75pt]  [color={rgb, 255:red, 80; green, 227; blue, 95 }  ,opacity=1 ]  {$\textcolor[rgb]{0.32,0.89,0.31}{X_{G}^{\varepsilon }( t)\overset{d}{=} X^{\varepsilon }( t) \ \ \ \ }$};
% Text Node
\draw (75.4,167.84) node [anchor=north west][inner sep=0.75pt]  [color={rgb, 255:red, 80; green, 227; blue, 100 }  ,opacity=1 ,rotate=-310.28] [align=left] {projection on};
% Text Node
\draw (138.36,90.82) node [anchor=north west][inner sep=0.75pt]  [color={rgb, 255:red, 80; green, 227; blue, 120 }  ,opacity=1 ,rotate=-313.04]  {$\rho _{t}^{\varepsilon }( y|x)$};
% Text Node
\draw (273.9,93.88) node [anchor=north west][inner sep=0.75pt]  [color={rgb, 255:red, 80; green, 227; blue, 84 }  ,opacity=1 ,rotate=-48.38]  {$\varepsilon \rightarrow 0$};
% Text Node
\draw (242.66,1.05) node [anchor=north west][inner sep=0.75pt]  [color={rgb, 255:red, 81; green, 227; blue, 80 }  ,opacity=1 ,rotate=-359.77] [align=left] {if GM works};
% Text Node
\draw (502,299.4) node [anchor=north west][inner sep=0.75pt]  [color={rgb, 255:red, 74; green, 144; blue, 226 }  ,opacity=1 ]  {$\textcolor[rgb]{0.82,0.01,0.11}{\overline{\rho }}\textcolor[rgb]{0.82,0.01,0.11}{(}\textcolor[rgb]{0.82,0.01,0.11}{x}\textcolor[rgb]{0.82,0.01,0.11}{)}\textcolor[rgb]{0.82,0.01,0.11}{=}\textcolor[rgb]{0.82,0.01,0.11}{\int }\textcolor[rgb]{0.82,0.01,0.11}{\rho }\textcolor[rgb]{0.82,0.01,0.11}{(}\textcolor[rgb]{0.82,0.01,0.11}{x,y}\textcolor[rgb]{0.82,0.01,0.11}{)}\textcolor[rgb]{0.82,0.01,0.11}{dy}$};
% Text Node
\draw (385.79,290.12) node [anchor=north west][inner sep=0.75pt]  [color={rgb, 255:red, 208; green, 2; blue, 27 }  ,opacity=1 ,rotate=-0.64] [align=left] {\textcolor[rgb]{0.82,0.01,0.11}{take marginal}};
% Text Node
\draw (411.64,133.2) node [anchor=north west][inner sep=0.75pt]  [color={rgb, 255:red, 208; green, 2; blue, 27 }  ,opacity=1 ,rotate=-359.65] [align=left] {if averaging is uniform in time};
% Text Node
\draw (274.64,338.2) node [anchor=north west][inner sep=0.75pt]  [color={rgb, 255:red, 0; green, 0; blue, 0 }  ,opacity=1 ,rotate=-359.65] [align=left] {if PM works};
% Text Node
\draw (417,385.4) node [anchor=north west][inner sep=0.75pt]  [color={rgb, 255:red, 0; green, 0; blue, 0 }  ,opacity=1 ]  {$\varepsilon \rightarrow 0$};
% Text Node
\draw (492,396.4) node [anchor=north west][inner sep=0.75pt]  [color={rgb, 255:red, 0; green, 0; blue, 0 }  ,opacity=1 ]  {$\nu =\overline{\rho }$};
% Text Node
\draw (547.64,345.2) node [anchor=north west][inner sep=0.75pt]  [color={rgb, 255:red, 0; green, 0; blue, 0 }  ,opacity=1 ,rotate=-359.65] [align=left] {if PM works};
% Text Node
\draw (156.64,470.2) node [anchor=north west][inner sep=0.75pt]  [color={rgb, 255:red, 0; green, 0; blue, 0 }  ,opacity=1 ,rotate=-359.65] [align=left] {if };
% Text Node
\draw (173,468.4) node [anchor=north west][inner sep=0.75pt]  [color={rgb, 255:red, 0; green, 0; blue, 0 }  ,opacity=1 ]  {$\lim _{\varepsilon \rightarrow 0} \rho ^{\varepsilon }( y|x) =\rho ^{( x)}( y)$};
\end{tikzpicture}
    \caption{Relations and dependencies between Gyöngy method (GM), projection method (PM) and averaging of slow-fast systems.}
    \label{fig:diagram}
\end{figure}
The commutativity of the red part of diagram in Figure \ref{fig:diagram} is by now well understood. In particular if  averaging is uniform in time then the limits $\varepsilon \rightarrow 0$ and $t\rightarrow \infty$ commute. Sufficient conditions for the averaging approximation to hold on infinite time horizons have been studied in \cite{crisan2022poisson} and in \cite{schuh2024conditions}. 

Looking at the green part of the diagram, if we apply the GM to a dynamics $Z_t\eep=(X_t\eep, Y_t\eep)$ with explicit time-scale separation as in \eqref{eq:sf},  then the law at time $t$ of $Z_t\eep$ depends on $\varepsilon$. Hence the same is true for the conditional law at time $t$, needed to apply the GM,  $\rho\eep_t(\cdot \vert x)$  and, in turn, for the dynamics produced via the GM,  $X_G\eep$. Assuming the GM succeeds in its intent - which will happen under the conditions of Proposition \ref{Prop: extension gyongy} - then  $X_G\eep(t)\stackrel{d}{=}X_t\eep$ (where $\stackrel{d}{=}$ is for equality in law). Letting now $\varepsilon$ to zero then necessarily produces a dynamics which coincides with the averaging approximation $X_A$.

Finally, following the black arrows,  starting from $(X_t\eep, Y_t\eep)$,  we can first apply the projection method - projecting the equation for $X_t\eep$ on the measure $\rho\eep(\cdot \vert x)$. Once the resulting dynamics $X_P\eep$ is obtained we can again let $\varepsilon$ to zero first and then $t\rightarrow \infty$ or vice versa. Letting $t$ to infinity first, and assuming $X_P\eep$ does converge to some limiting law $\nu\eep$,  if the PM has produced a dynamics with the correct long time behaviour (which is guaranteed under the assumptions of Theorem \ref{thm:projection method works}), then $\nu\eep=\bar\rho\eep \rightarrow \bar\rho $ as $\varepsilon \rightarrow 0$. If instead we let $\varepsilon$ to zero first, then the resulting process $X_P$ will coincide with the averaged dynamics $X_A$ if and only if $\lim_{\varepsilon \rightarrow 0} \rho\eep(\cdot \vert x) = \rho^{(x)}(\cdot)$. We study this limit in Section \ref{sec: Carsten's section}.
\\
We conjecture that the same conditions under which uniform in time averaging holds, are sufficient to guarantee commutativity of the whole diagram. But we do not address this matter here.

%%%%%%%%%%%%%%%%%%%%%
%%%%%%%%%%%%%%%%%%%%%%%%%%%
%%%%%%%%%%%%%%%%%%%%%%%%%%%%%
%%%%%%%%%%%%%%%%%%%%%%%%%

%%%%%%%%%%%STOP STOP

%
%In some sense, I would make all these points somewhere in the intro but then also emphasize that somehow we are not specifically concerned with them, but rather with the mathematical consequences of deciding to use one method rather than the other, so to speak. I.e. we compare these procedures from a strictly mathematical point of view without entering the discussion of which of these methods is more physical/more adequate in a given application. These two methods are at different stages of mathematical development, with averaging theory being by far more developed than the theory behind the PM. One of the purposes of this paper is to consolidate the theory behind the PM.  

\subsubsection*{Outline}
With this in mind, the rest of the paper is organised as follows. In Section \ref{sec: sec 3} we first give more background on the GM and prove Proposition \ref{Prop: extension gyongy}, which is a slight extension of the results in \cite{gyongy1986mimicking}, and then provide the heuristic proof of Theorem \ref{thm:projection method works}. In Section \ref{sec: rhox and rhox|y} we study the relation between the invariant measure of the auxiliary dynamics and the ECD; specifically, in Section \ref{sec: rhox equals rhox|y} we provide sufficient conditions under which  equality \eqref{eqn:frozen=conditional} holds, see Proposition \ref{prop: equivalent conditions}, and then show examples of systems where such an equality does not hold; in Section \ref{sec: Carsten's section} we identify conditions under which the PM and the averaging approximation coincide in the limit of time scale separation -- that is, we study the convergence of $\rho\eep(\cdot \vert x)$ to $\rho^{(x)}$ as $\varepsilon\rightarrow 0$. In Section \ref{sec: statement and proof of Thm 6.2} we first give background on the long time behaviour of non-autonomous SDEs and slightly generalise the results of \cite{kunze2010nonautonomous, cass2021long}; then provide the rigorous statement and proof of Theorem \ref{thm:projection method works}. Section \ref{sec: statement and proof of Thm 6.2} also contains a more thorough discussion of Example \ref{example:short version}, see Subsection \ref{sec: counterexample to Gyongy}.

%We will provide detail about this example in Section \ref{sec: statement and proof of Thm 6.2}, see Example \ref{example: counterexample to Gyongy}.  However, in summary, in this case 
%$\rho(\cdot \vert x) \sim \mathcal N (x, 1/2)$ while $\rho^{(x)} \sim \mathcal N(0, 1/2)$. As a consequence, the effective dynamics produced via the PM and the one produced via averaging are, respectively
%$$
%dX_P(t)= 0, \quad X_P(0)=X_0
%$$
%and $$d\bar X_t = -\bar X_t, \quad \bar X(0)=X_0 \,.$$
% as the marginal $\bar \rho$ is here given by $\bar\rho \sim \mathcal N(0, 1/2)$. 
  
\section{Heuristics on the proof of Theorem \ref{thm:projection method works} and some results regarding the Gy\"ongy method}\label{sec: sec 3}
In this section we discuss two different approaches to studying the asymptotic properties of the coarse grained dynamics \eqref{eq:sfgenaralprojection} obtained from \eqref{eq:sfep=1} via the PM. The first, described  in Subsection \ref{sec:first approach} and summarised in Proposition \ref{lem:bar rho invariant for Lp}, is substantially known, in the sense that Proposition \ref{lem:bar rho invariant for Lp} is  special case of \cite[Proposition 4]{zhang2016effective} (when the coarse-graining map $\xi$ in that paper corresponds to simple projection  on the $x$-variable).  We recap it here for context and because the calculation in its proof will be useful for later discussion. The second, presented in Subsection \ref{sec:heuristic explanation of theorem}, is instead new to this paper and it is the one leading to Theorem \ref{thm:projection method works}. 
In particular, 
Subsection \ref{sec:heuristic explanation of theorem} contains a heuristic explanation of Theorem \ref{thm:projection method works} and of its method of proof. The actual proof is postponed to Section \ref{sec: statement and proof of Thm 6.2}.   Subsection \ref{subsec:proof gyongy} contains instead some results on the GM which are used in the actual proof of Theorem \ref{thm:projection method works}   -- see in particular  Proposition \ref{Prop: extension gyongy} which is a slight extension of the results of Gy\"ongy,   and  Note \ref{note: on gyongy extension}, where we comment on the relationship between our Proposition \ref{Prop: extension gyongy} and the results of \cite{gyongy1986mimicking}.

\subsection{First Approach}\label{sec:first approach}

Here we show that the process $X_P$ constructed via the PM has the desired invariant measure.  
Before stating the next proposition we recall that a measure $\mu$ is (infinitesimally) invariant for a Markov dynamics say on $\R^n$ generated by an operator $\mathfrak L$ (in short, invariant for $\mathfrak L$) if 
\begin{equation}\label{def:infinitesimally invariant}
    \int_{\R^n} (\mathfrak L f)(x) \mu(dx) = 0 
\end{equation}
for every $f$ in the domain of $\mathfrak L$, see e.g. \cite{lorenzi2006analytical}. Moreover, 
the infinitesimal generator $\cL$ of system \eqref{eq:sfep=1} is the operator defined on smooth functions as 
\begin{equation}\label{eq:gen}
    \cL=\cL_x + \cL_y
\end{equation}
where $\cL_x$ and $\cL_y$ are, respectively, the generators associated to the slow and fast component of the dynamics \eqref{eq:sf}; that is, 
\begin{subequations}
    \begin{alignat}{2}\label{eq:genx}
    \cL_x & = \frac{1}{2}\alpha\alpha^T\colon\nabla_x^2 && + f\cdot\nabla_x\\\label{eq:geny}
    \cL_y & = \frac{1}{2}\beta\beta^T\colon\nabla_y^2 && + g\cdot\nabla_y \,.
    \end{alignat}
\end{subequations}
The (flat) $L^2$-adjoint of $\cL$ is denoted by  $\cL'$ and it is given by 
\begin{equation}\label{flat adjoint}
    \cL' = \cL_x' + \cL_y'\,.
\end{equation}
Note that, because of the structure of the noise in \eqref{eq:sfep=1},  $\cL_x$ and $\cL'_x$ contain only derivatives with respect to $x$, whereas $\cL_y$ and $\cL'_y$ are differential operators in $y$ only.

We will also use the fact that the generator of the  evolution \eqref{eq:sfgenaralprojection}, is  given by the second order differential operator $\cL_P$ defined as 
\begin{equation}\label{eqn:generator Lp}
\cL_P = \frac12 \alpha_P(x) :\nabla_x^2 + f_P(x) \cdot \nabla_x \, .
\end{equation}
\begin{prop}
\label{lem:bar rho invariant for Lp}
    Let $\cL_P$ be the generator of the dynamics \eqref{eq:sfgenaralprojection}, defined as  in \eqref{eqn:generator Lp},  and  $\cL$ be the generator of  \eqref{eq:sfep=1}, defined  as in \eqref{eq:gen}. Assume $\cL$ is uniformly elliptic, i.e. that \eqref{eqn:unif_ellipticity_initial_system} holds, 
    with smooth coefficients and that  $\rho$ is a probability measure on $\R^n$, infinitesimally invariant for  $\mathcal L$.  Then $\bar\rho$ (the marginal of $\rho$ defined  in \eqref{eqn:bar rho def}) is infinitesimally invariant for $\cL_P$; that is,  
    \begin{equation}\label{eqn:bar rho invariant for Lp}
        \int_{\R^d} (\cL_P h)(x) \bar{\rho}(x) \, dx = 0 
    \end{equation}
     for every $h$ in the domain of $\cL_P$. 
    \end{prop}

 Let us point out that this proposition  only relates the stationary states of  $\cL_P$ to the stationary states of $\cL$, but it gives no information on the dynamics. The  line of thought that we will present in the Subsection \ref{sec:heuristic explanation of theorem} is more informative of the dynamics as well.   See also Subsection \ref{sec: counterexample to Gyongy} on this point. 
\begin{proof}[Proof of Proposition  \ref{lem:bar rho invariant for Lp}]
First note that if $\cL$ is uniformly elliptic and with smooth coefficients then any invariant measure for $\cL$ has a smooth density and is strictly positive. Such properties are then inherited by both the marginal $\bar \rho$ and the ECD $\rho(\cdot\vert x)$.  With a straightforward calculation this implies that also $\cL_P$ is uniformly elliptic. \footnote{To see this, letting $v=(w,u)\in \R^d\times\R^{n-d}$, from \eqref{eqn:unif_ellipticity_initial_system} we deduce $\langle \alpha\alpha^T(x,y)w, w\rangle\geq \zeta |w|^2$ for every $(x,y) \in \R^d\times\R^{n-d}, w \in \R^d$.  Using this fact,  the expression for $\alpha_P$ and \eqref{eqn:conditional is a prob measure } the uniform ellipticity of $\cL_P$ follows. }

With this premise, to prove the statement  it suffices to verify that \eqref{eqn:bar rho invariant for Lp} holds for every $h \in C_c^{\infty}(\R^d)$, see \cite[Corollary 8.1.7]{lorenzi2006analytical}. To this end,   using the fact that $\rho(\cdot\vert x)$ is a probability measure for every $x \in \R^d$, i.e.
\begin{equation}\label{eqn:conditional is a prob measure }
    \int_{\R^{n-d}}\rho(y\vert x) dy = 1 \qquad \mbox{for every } x \in \R^d \, 
\end{equation}
and the following basic relation between joint, conditional and marginal distribution
\begin{equation}\label{eqn:joint=conditional times marginal}
    \rho(x,y) = \rho(y\vert x) \bar{\rho}(x) \, , 
\end{equation}
    we have (writing everything, to ease notation and  without loss of generality,  with $d=n-d=1$ ):
    \begin{align*}
        \int_{\R^d} (\cL_P h)(x) \bar{\rho}(x) dx  &= \int_{\R^d} f_P(x)  (\partial_x h) \bar\rho(x) dx+\frac{1}{2}\int_{\R^d} (\alpha_P(x)\partial_{xx}h) \bar\rho(x) dx\\   &\stackrel{\eqref{eqn:projection-coefficients}, \eqref{eqn:conditional is a prob measure }}{=}
        \int_{\R^d}  (\partial_xh)\bar{\rho}(x) \,dx\int_{\R^{n-d}}  f(x,y) \rho(y\vert x)\, dy \\
        & + \frac{1}{2}\int_{\R^d} (\partial_{xx}h) \bar\rho(x) dx \int_{\R^{n-d}}  
 \alpha^2(x,y) \rho(y\vert x) dy\\
        & \stackrel{\eqref{eqn:joint=conditional times marginal}}{=} \int_{\R^n} f(x,y) \rho(x,y) \pa_xh \,  dx dy+ 
        \frac{1}{2}\int_{\R^n}   \rho(x,y) \alpha^2(x,y) \pa_{xx} h \, dx dy \\
        & = \int_{\R^n} (\cL_x h) (x,y) \rho(x,y) dx dy, 
    \end{align*}
    where $\cL_x$ is as in \eqref{eq:genx}. 
    We now want to show  that the RHS of the above is equal to zero. 
    Using  \eqref{eq:gen}  and  the invariance of $\rho$ under $\cL$,   we can write
    \begin{align*}
   0=  \int_{\R^{n}} (\cL \varphi)(x,y) \rho(x,y) dxdy & = \int_{\R^{n}} (\cL_x \varphi)(x,y) \rho(x,y) dxdy\\
    &+ \int_{\R^{n}} (\cL_y \varphi)(x,y) \rho(x,y) dxdy \,, 
    \end{align*}
    for every $\varphi$ in the domain of $\cL$. 
Functions $\varphi$ of the form $\varphi(x,y)=h(x) \mathbf{1}(y)$ with $h \in C_c^{\infty}(\R^d)$  and  $\mathbf 1(y)=1$ for every $y$,  do belong to the domain of $\cL$. If we apply the above to such functions then, since $\cL_y$ is a differential operator in the $y$-variable only, the second addend on the RHS of the above vanishes, and we are left with 
    \begin{equation}\label{zero on functions of x}
        \int_{\R^{n}} (\cL h)(x,y) \rho(x,y) dxdy = \int_{\R^{n}} (\cL_x h)(x,y) \rho(x,y) dxdy = 0 \,.
    \end{equation}
This concludes the proof. 
\end{proof}

\subsection{Heuristic explanation of Theorem \ref{thm:projection method works}}\label{sec:heuristic explanation of theorem} 

Here we explain how the Gy\"ongy approximation $X_G$ can be instrumental in proving the asymptotic properties of the process $X_P$.
We explain here the general reasoning, which is the backbone of the proof of Theorem \ref{thm:projection method works}, unencumbered by technical matters. Comments on technical aspects are contained in Note \ref{noteappendix:note on assumption non autonomous SDEs} 
and Note \ref{note:on theor3_5}. 

We start by noting that the evolution  \eqref{eqn:gyongy-coarse-graining}  is non-autonomous, as the coefficients $f_G(t,x)$ and $\alpha_G(t,x)$ depend explicitly on time through the conditional law $\rho_t(\cdot\vert x)$. On this point we  clarify that in principle one could also regard  \eqref{eqn:gyongy-coarse-graining} as a non-linear SDE of McKean-Vlasov type. Whether we interpret \eqref{eqn:gyongy-coarse-graining} as being non-autonomous (and linear in the sense of McKean) or a form of non-linear McKean Vlasov evolution depends on whether we consider $\rho_t(\cdot \vert x)$ to be part of the unknown or not. Indeed in principle one can obtain the conditional distribution $\rho_t(\cdot\vert x)$ by solving system \eqref{eq:sfep=1} and then use it to calculate the coefficients $f_G(t,x), \alpha_G(t,x)$. With this point of view one can regard \eqref{eqn:gyongy-coarse-graining} as a non-autonomous evolution (linear in the sense of McKean).  If we were to view $\rho_t(\cdot\vert x)$ as an unknown in \eqref{eqn:gyongy-averaged-coefficients}-\eqref{eqn:gyongy-averaged-coefficients2},  then \eqref{eqn:gyongy-coarse-graining} would be non-linear in the sense of  McKean. This might prove to be a fruitful point of view for simulations, but it is not the one that we take in the reasoning to come and in the proof of Theorem \ref{thm:projection method works} - in what comes next we use the  results of \cite{angiuli2013hypercontractivity, cass2021long} in an essential way and such results are not stated for McKean-Vlasov type equations, so they would not apply. 

Let us now move on to informally describing the results of \cite{angiuli2013hypercontractivity, cass2021long}, see in particular \cite[Section 6]{angiuli2013hypercontractivity} and \cite[Section 6]{cass2021long}. 

Consider an SDE in $\R^d$ of the form 
\begin{align}
    d\cX_t &= b(t, \cX_t) dt + \sigma(t, \cX_t) dU_t \,, 
    \label{generalSDEnon-autonomous} 
    \end{align}
which we can rewrite in Stratonovich form as     
    \begin{align}
  d\cX_t &= B(t, \mathcal X_t) dt + \sigma(t, \cX_t) \circ dU_t \nonumber\\
  & = B(t, \mathcal X_t) dt + \sum_{i=1}^d\sigma\col (t, \cX_t) \circ dU^i_t\,,  \label{stratonovich non-autonomous general SDE} 
\end{align}
where $\sigma\col\colon\R_+\times\R^d \rightarrow \R^d$ is the vector field given by the  $i$-th column of the $d\times d$ matrix $\sigma$,  $B\colon\R_+ \times \R^d\rightarrow \R^d$ is the vector with $i$-th component equal to 
$$
B^i(t,x)= b^i(t,x) - \frac12\sum_{j,k=1}^d \sigma_{jk}\frac{\pa\sigma_{ik}}{\pa x_j}, \quad i=1, \dots, d, 
$$
and  $U^i_t$ is the one-dimensional Brownian motion corresponding to the $i$-th component of the vector $U_t$. \footnote{Strictly speaking the work \cite{angiuli2013hypercontractivity}, or at least the parts we will use, consider diffusion coefficients that do not depend on the state variable, i.e. of the form $\sigma=\sigma(t)$. Extensions have however been considered in \cite{cass2021long}. We will be more precise on this in the proof of Theorem \ref{thm:projection method works}.}
    
The SDE \eqref{generalSDEnon-autonomous} (or, equivalently, \eqref{stratonovich non-autonomous general SDE}) is non-autonomous. In this case one cannot expect the long time behaviour of the dynamics to be described by a single invariant measure, at least not in general (see Appendix \ref{appendix:long time behaviour of non-autonomous SDEs} for more background on this). However, suppose we know that the coefficients $B(t,x), \sigma(t,x)$ are such that 
\begin{equation}\label{eqn:limit coefficients - informal}
    \lim_{t\rightarrow \infty} B(t,x) = \bar{B}(x), \quad \lim_{t\rightarrow \infty} \sigma(t,x) = \bar{\sigma}(x) \,.
\end{equation}
In this case we can consider the autonomous SDE with coefficients given by $\bar B$ and $\bar \sigma$, i.e. the dynamics
\begin{equation}\label{stratonovich autonomous limiting SDE}
    d\bar \cX(t) = \bar B(\bar \cX_t) dt + \bar \sigma(\bar \cX_t) \circ  dU_t \,.
\end{equation}
Then, under  assumptions on the coefficients $B$ and $\sigma$, one expects that the long time behaviour of  $\cX_t$ and $\bar\cX_t$ will be the same. This has been shown to be the case in \cite{angiuli2013hypercontractivity} and later in \cite{cass2021long}, under milder assumptions. 

This result can be applied in a straightforward way to the dynamics \eqref{eqn:gyongy-coarse-graining} and \eqref{eq:sfgenaralprojection}, as \eqref{eqn:gyongy-coarse-graining} is of the form \eqref{generalSDEnon-autonomous} and  \eqref{eq:sfgenaralprojection} is precisely the corresponding `limit equation' as in \eqref{stratonovich autonomous limiting SDE}.  Indeed, modulo caveats and subtleties, the intuition behind Theorem \ref{thm:projection method works} is as follows: if $\rho_t(x,y)$ converges to $\rho(x,y)$ as $t\rightarrow \infty$, then also  $\bar \rho_t(x) \rightarrow \bar \rho(x)$ and hence $\rho_t(y\vert x)\rightarrow \rho(y\vert x)$ as $t\rightarrow \infty$. In turn, under good conditions,   this implies convergence of the associated coefficients, i.e.
\begin{equation}\label{convergemce stratonovich coefficients}
\lim_{t\rightarrow \infty} F_G(t,x) = F_P(x), \quad \lim_{t\rightarrow \infty} \alpha_G(t,x) = \alpha_P(x) \,, \quad \mbox{for every } x \in \R^d, 
\end{equation}
where $F_G$ and $F_P$ are the Stratonovich drifts of \eqref{eqn:gyongy-coarse-graining} and \eqref{eq:sfgenaralprojection}, respectively. 
Hence, again under appropriate assumptions, we deduce that $X_G(t)$ has the same long time behaviour as $X_P(t)$. However from \eqref{eqn:law equality gyongy} we already know that $X_G(t)$ has the same long-time behaviour as $X(t)$. We can therefore conclude that the invariant measure of $X_P(t)$ is precisely $\bar\rho$. This is the backbone of the reasoning which allows us to conclude and prove Theorem \ref{thm:projection method works}.

Note that while the justification of the projection method given for gradient systems relies on the knowledge of the explicit form for the invariant measure $\rho$, the reasoning   we have outlined above (as well as the statements and proof of Theorem \ref{thm:projection method works}) requires no such knowledge, coherently with the fact that for general non gradient systems  one would not expect to explicitly know  the invariant measure $\rho$. 

\subsection{On the Gy\"ongy method.}\label{subsec:proof gyongy}  
Here we want to show that the equality \eqref{eqn:law equality gyongy} holds.  Recalling that $\rho_t$ is the law of the solution $(X_t, Y_t)$ of system \eqref{eq:sfep=1}, $\rho_t$ is the unique weak solution to the Fokker-Planck equation
$$
\pa_t \rho_t = \cL'\rho_t = \cL_x'\rho_t +\cL_y'\rho_t \, ,  
$$
where $\cL_x', \cL_y'$ are as in \eqref{flat adjoint}.  We can therefore  write
\begin{align*}
    \frac{d}{dt}\int_{\R^d} dx\,  \bar{\rho}_t(x) \varphi(x) &\stackrel{\eqref{joint and marginal time t}}{=} 
    \frac{d}{dt} \int_{\R^d}\!\!\int_{\R^{n-d}}  \varphi(x)  \rho_t(x,y)\,  dx \, dy\\
    & = \int_{\R^d}\!\!\int_{\R^{n-d}} \varphi(x) \cL_x' \rho_t(x,y) \,dx \, dy+ \int_{\R^d}\!\!\int_{\R^{n-d}}  \varphi(x)\cL_y' \rho_t(x,y) \,dx \, dy, 
\end{align*}
for every test function $\varphi \in C_0^{\infty}(\R^d)$.  
Let us now look at the second addend in the above. Since $\cL_y'$ is a differential operator in the $y$ variable only while $\varphi$ is a function of $x$ only (and hence comes out of the inner integral),  if the functions $g(x, \cdot)\rho_t(x,\cdot)$ and $\pa_{y_j}(\beta\beta^T(x, \cdot)\rho_t(x, \cdot))$ vanish at infinity for each $x \in \R^d$ and each $j=1, \dots, n-d$ (this is guaranteed for example if they are in $H^1(\R^{n-d})$ for each $x \in \R^d$, more comments in this in Note \ref{note: on gyongy extension}), the inner integral of the second addend in the above vanishes. Indeed, under such an assumption, we have    
\begin{align*}
\int_{\R^{n-d}} \cL_y' \rho_t(x,y)\, dy = 
 \int_{\R^{n-d}} \sum_i \pa_{y_i} \!\left[ (g_i(x,y) \rho_t(x,y))
 + \sum_j \pa_{y_j}((\beta\beta^T(x,y))_{ij} \rho_t(x,y))\right]dy = 0
\end{align*}
for every $x \in \R^d$. 
Hence
\begin{align*}
    \frac{d}{dt}\int_{\R^d}   \bar{\rho}_t(x) \varphi(x) \, dx
    & = \int_{\R^d}\!\!\int_{\R^{n-d}} \varphi(x) \cL_x' \rho_t(x,y) \, dx\, dy\\
   & = 
    \iint  \varphi(x)\sum_i\pa_{x_i}\left[(f(x,y) \rho_t(x,y)) 
    +\frac{1}{2}  \sum_j\pa_{x_j} ((\alpha\alpha^T)_{ij}\rho_t(x,y))
    \right]\, dx \, dy\\
    &\stackrel{\eqref{joint and marginal time t}}{=} - \sum_i \iint  \pa_{x_i}\varphi(x) \,  f(x,y) \bar\rho_t(x) \rho_t(y\vert x) dx \, dy \\
    & +
    \frac{1}{2} \sum_{ij}\iint (\alpha\alpha^T)_{ij}(x,y) \pa_{x_ix_j} \varphi(x) \bar\rho_t(x) \rho_t(y\vert x) \, dx \, dy\\
    & = - \sum_i \int_{\R^d}  \pa_{x_i}\varphi(x) \bar\rho_t(x) f_G(t,x) \, dx\\
    &+ \frac{1}{2} \sum_{ij} \int_{\R^d}   \pa_{x_ix_j} \varphi(x) \bar\rho_t(x) (\alpha_G)_{ij}(t,x) \, dx \,.
\end{align*}
From the first and last line of the above (and after an integration by parts in $x$) one can read off the weak formulation of the equation for $\bar\rho_t$, and see that this coincides with the weak formulation of the  Fokker-Planck equation for the evolution \eqref{eqn:gyongy-coarse-graining}. Since the solution of the Fokker-Plank equation is the law of the corresponding process,  the following proposition holds.

\begin{prop}\label{Prop: extension gyongy}
    Assume the following: 
    
   i) The SDEs \eqref{eq:sfep=1}
and \eqref{eqn:gyongy-coarse-graining} have a unique weak solution and  the laws of the solutions of both SDEs have strictly positive densities for each time $t\geq 0 $. 

ii) The density $\rho_t$ of the process $(X_t, Y_t)$ solution of \eqref{eq:sfep=1} is such that the functions $g(x,\cdot) \rho_t(x, \cdot)$ and $\pa_{y_j}(\beta\beta^T(x, \cdot)\rho_t(x,\cdot))$ vanish at infinity for each $x \in \R^d, t\geq 0$ fixed and each $j=1, \dots, n-d$.

Then the marginal $\bar\rho_t$  of the joint density $\rho_t$ of the law of the process $(X_t,Y_t)$, defined in \eqref{joint and marginal time t}, coincides with the law of $X_G(t)$ (solution of \eqref{eqn:gyongy-coarse-graining}) for each time $t\geq 0$; that is, \eqref{eqn:law equality gyongy} holds. 
\end{prop}

%%%%%%%%%%%%%%%%%%%%%%%%%%%
%%%%%%%%%%%%%%%%%%%%%%%%%%%%%%%%%%%%%%NOTE

\begin{note}\label{note: on gyongy extension}
{\textup{Proposition \ref{Prop: extension gyongy} gives sufficient conditions in order for the coarse grained dynamics \eqref{eqn:gyongy-coarse-graining} proposed by G\"yongy to have the desired property \eqref{eqn:law equality gyongy}. Let us make some   comments on the assumptions of  Proposition \ref{Prop: extension gyongy}, comparing them to those in \cite{gyongy1986mimicking}. 
\begin{itemize}
\item Assumption ii) in Proposition \ref{Prop: extension gyongy} is standard. It is satisfied  assuming the process \eqref{eq:sfyep=1} is elliptic and it  enjoys good ergodic properties (for each $x$ fixed). Sufficient conditions (of Lyapunov type) for this to hold have been extensively studied, see e.g. \cite{bogachev2006global}, or can also be proved `by hand' using classical approaches such as in \cite{butta2018non, degond1986global}. 
\item Condition i) cannot be dispensed with. When it fails the result may simply not hold. The system in Subsection \ref{sec: counterexample to Gyongy} illustrates this fact.  
\item Condition i) imposes a requirement on \eqref{eq:sfep=1} and on its associated Gy\"ongy approximation \eqref{eqn:gyongy-coarse-graining}.  Since the coefficients of \eqref{eqn:gyongy-coarse-graining} are constructed starting from the coefficients of \eqref{eq:sfep=1}, it would be desirable to state general conditions on the coefficients of \eqref{eq:sfep=1}  which suffice to guarantee the well posedness of both \eqref{eq:sfep=1} and \eqref{eqn:gyongy-coarse-graining}. If the coefficients of \eqref{eq:sfep=1} are bounded and smooth and \eqref{eq:sfep=1} is elliptic then \eqref{eqn:gyongy-coarse-graining} inherits the same properties and  the first requirement of Proposition \ref{Prop: extension gyongy} is satisfied. Indeed,  let us recall that elliptic SDEs with bounded and smooth coefficients are weakly well posed, in the sense that they admit a unique weak solution which has a strictly positive density, see \cite{karatzas2014brownian}. If $f$ and $\alpha$ in \eqref{eq:sfep=1} are bounded, then so are the coefficients $f_G, \alpha_G$ of \eqref{eqn:gyongy-coarse-graining}:
$$
f_G(x) = \int_{\R^{n-d}} f(x,y) \rho_t(y \vert x) \, dy \leq 
\sup_{x,y} |f| \int_{\R^{n-d}}  \rho_t(y \vert x) \, dy = \sup_{x,y} |f|,  
$$
because $\rho_t(\cdot \vert x)$ is a probability distribution for every $t, x$ fixed, analogously for $\alpha_G$. 
    Moreover, with the same reasoning as the one at the start of  the proof of Proposition \ref{lem:bar rho invariant for Lp}, it is easy to see that if \eqref{eq:sfep=1} is uniformly elliptic and with smooth coefficients the same is true of \eqref{eqn:gyongy-coarse-graining};  hence both  the  law of the  solution of \eqref{eq:sfep=1}  and the law of the solution of \eqref{eqn:gyongy-coarse-graining} admit a strictly positive density.  
    Nonetheless outside of the context of smooth bounded coefficients, it is not at all straightforward to find conditions on \eqref{eq:sfep=1} which imply the desired properties for \eqref{eqn:gyongy-coarse-graining}. We comment on this next. 
    \item In this section we have presented the results of \cite{gyongy1986mimicking} in a setting which is adapted to the one of the present paper. Nonetheless the setup of \cite{gyongy1986mimicking} and of subsequent related literature, see e.g. \cite{brunick2013mimicking} and  references therein,  is much more general. In \cite{gyongy1986mimicking, brunick2013mimicking} the authors consider a random process, say $x_t$, which is not assumed to necessarily  be the solution of an SDE, it is simply assumed to have an Ito differential, i.e. $x_t$ it is a process of the form
    \begin{equation}\label{eqn:ito differential}
    dx_t = \gamma(t, \omega) dt+ \delta (t, \omega) dW_t 
    \end{equation}
    where $\omega$ is the realization in the underlying probability space. 
    The purpose of \cite{gyongy1986mimicking, brunick2013mimicking} is to provide a way to construct  another process, say $z_t$, which solves an SDEs and that inherits certain properties of $x_t$. For example, like in our case, the process $z_t$ might be required to preserve certain  marginals of $x_t$.  In \cite{gyongy1986mimicking, brunick2013mimicking} it is explicitly noted that, at least in the elliptic setting, assuming the Gy\"ongy procedure applied to $x_t$ gives rise to an SDE for $z_t$ which is well posed,   then the law of  the desired marginals of $x_t$ and the one for $z_t$ will coincide. According to \cite{gyongy1986mimicking} if the coefficients $\gamma, \delta$ are bounded and measurable and $\delta$ is uniformly elliptic,  both \eqref{eqn:ito differential} and the associated coarse grained equation $z_t$ have a unique weak solution  and moreover  $z_t$ will enjoy the desired property of mimicking the marginal of $x_t$. Boundedness and ellipticity are sufficient conditions, but they are not necessary. Nonetheless various authors have explicitly pointed out that,   outside of this bounded elliptic case it is truly difficult to state conditions directly on the coefficients  of \eqref{eqn:ito differential} such that $z_t$ is well posed, see e.g.~\cite{brunick2013mimicking} for a thorough discussion on this. The work \cite{brunick2013mimicking} also produced examples where the SDE for $z_t$  has multiple solutions, none of them having the desired law.
     On the whole, the problem of finding conditions in terms of the coefficients of \eqref{eqn:ito differential} such that the associated coarse grained equation is at least well posed   seems to be still open.   In our setting, however,  the process we want to coarse grain   is already the solution of an SDE rather than a general Ito differential.  Due to this additional structure it should be an approachable problem to find assumptions (more general than  boundedness and ellipticity) on the coefficients of \eqref{eq:sfep=1} which guarantee  well posedness and existence of a strictly positive density for  \eqref{eqn:gyongy-coarse-graining}. This will be the subject of future work.   
\end{itemize}
}}
\end{note}

%%%%%%%%%%%%%%%%%%%%%%%%%%%%%5
%%%%%%%%%%%%%%%%%%%%%%%%%%%5
%%%%%%%%%%%%5HEURISTIC

\section{Invariant measure of the auxiliary dynamics vs. ECD} \label{sec: rhox and rhox|y}
In this section, we  consider systems of SDEs which are a bit more general than system \eqref{eq:sfep=1}, namely we consider systems of the form
\begin{subequations}
\label{eq:general SDE}
\begin{align}   
dX_t&=f(X_t, Y_t)\,dt+\alpha_{11}(X_t,Y_t)dU_t+\alpha_{12}(X_t,Y_t)\, dW_t,\\
dY_t&=g(X_t, Y_t)\,dt+\beta(X_t,Y_t)\, dW_t,
\end{align}
\end{subequations}
%    d\begin{pmatrix}
%        X_t\\Y_t
%    \end{pmatrix}=\begin{pmatrix}
%        b_1(X_t,Y_t)\\
%        b_2(X_t,Y_t)
%    \end{pmatrix}\,dt +\sqrt{2} \begin{pmatrix}
%        \sigma_{11}(X_t,Y_t)& \sigma_{12}(X_t,Y_t)\\
%        \sigma_{21}(X_t,Y_t)&\sigma_{22}(X_t,Y_t)
%    \end{pmatrix}\,\begin{pmatrix}
%        dU_t\\ dW_t
%    \end{pmatrix},
%\end{equation}
%\todo{We should homogenise the notation, e.g. drift $f,g$ etc.}
where, $(X_t,Y_t)\in \R^{d}\times \R^{n-d}$,  $U_t, W_t$ are independent standard  Brownian motions in $\R^{d}$ and $\R^{n-d}$, respectively, and the coefficients 
\begin{align*}
&f: \R^{d}\times \R^{n-d}\rightarrow \R^{d},~ g\colon\R^{d}\times \R^{n-d}\rightarrow \R^{n-d}, \\    
&\alpha_{11}\colon\R^{d}\times \R^{d}\rightarrow \R^{d\times d},~\alpha_{12}\colon\R^{d}\times \R^{n-d}\rightarrow \R^{d\times (n-d)},~\beta: \R^{d}\times \R^{n-d}\rightarrow \R^{(n-d)\times (n-d)},
\end{align*}
are assumed smooth; we also  assume throughout that system \eqref{eq:general SDE} is uniformly elliptic.
As in previous sections we consider the family of auxiliary processes $Y^{(x)}_t$ obtained by freezing the value of $X_t$ in the equation for $Y_t$,  
\begin{equation}\label{eqn:general auxiliary process}
    dY_t^{(x)}= g(x, Y_t^{(x)}) dt + \beta(x,Y_t^{(x)}) dW_t \,  
\end{equation}
and  make the standing assumption that the process $Y_t^{(x)}$ admits, for each $x \in \R^d$, a unique invariant measure, $\rho^{(x)}$ (which is a measure on $\R^{n-d}$) and also that the full system \eqref{eq:general SDE} admits a unique invariant measure $\rho$. As in previous sections, the $x-$marginal of $\rho$ and the associated conditional measure are denoted $\bar\rho$ and $\rho( \cdot\vert x)$, see \eqref{eqn:bar rho def} and \eqref{eqn:conditional for system 1}.

\subsection{Necessary and sufficient conditions}% for $\rho^{(x)}=\rho(\cdot\vert x)$}
\label{sec: rhox equals rhox|y}

With this premise, the main purpose of this section is to seek necessary and sufficient conditions under which $\rho^{(x)}$ and $\rho(\cdot\vert x)$ coincide, for every $x \in \R^d$.  The main result of this section is Proposition \ref{prop: equivalent conditions} below.

We note that in \eqref{eq:general SDE} the noise in the equations for $X_t$ is not independent of the noise in the equation for  $Y_t$. This is in contrast to system \eqref{eq:sfep=1}.    We define
\[
A(x,y)=\begin{pmatrix}
    A_{11}(x,y)& A_{12}(x,y)\\
    A_{21}(x,y)& A_{22}(x,y)
\end{pmatrix},
\]
where
\begin{align*}
&A_{11}(x,y)=(\alpha_{11}\alpha_{11}^T+\alpha_{12}\alpha_{12}^T)(x,y), \quad A_{12}(x,y)=(\alpha_{12}\beta^T)(x,y),
\\& A_{21}(x,y)=A_{12}^T(x,y),\quad A_{22}(x,y)=(\beta\beta^T)(x,y).
\end{align*}
Then the $L^2$-adjoint of the  generator of \eqref{eq:general SDE} is given by
\begin{equation}
 \label{eq: full generator}   
\mathcal{L}'\nu=\mathcal{L}'_x\nu+\mathcal{L}'_{xy}\nu+\mathcal{L}'_y \nu,
\end{equation}
where
\begin{align*}
\mathcal{L}'_x\nu&:=-\mathrm{div}_x(f(x,y)
\nu)+\mathrm{div}_x[\mathrm{div}_x(A_{11}(x,y)\nu)],\\ 
\mathcal{L}'_{xy}\nu&=\mathrm{div}_{x}[\mathrm{div}_{y}(A_{12}(x,y)\nu)]+\mathrm{div}_{y}[\mathrm{div}_{x}(A_{21}(x,y)\nu)],\\
\mathcal{L}'_y\nu&:=-\mathrm{div}_y(g(x,y)\nu]+\mathrm{div}_y[\mathrm{div}_y(A_{22}(x,y)\nu)].
\end{align*}
The following proposition is the main result of this section.
%\textcolor{teal}{I don't fully understand the proposition. 
%I am wondering whether $\cL'(\rho)=0$ should be a ``global'' assumption, such that under the assumption that $\rho(x,y)=\bar{\rho}(x)\rho(y|x)$ is a (the?) positive solution to $\cL'(\rho)=0$, we have that $\rho^{(x)}(y)=\rho(y|x)$ and \eqref{eq:condition} are equivalent.}
\begin{proposition}
\label{prop: equivalent conditions}
If $\mathcal{L}'\rho=0$ and $\rho^{(x)}(y)=\rho(y|x)
$, then
\begin{equation}  
\label{eq:condition}
\mathcal{L}_x'\rho=-\mathcal{L}'_{xy}\rho~ \text{and}~ \mathcal{L}'_y(\rho)=0.
\end{equation}
Vice versa, if $\rho$ satisfies \eqref{eq:condition} and $\bar\rho$ is its $x$-marginal density, assuming $\bar\rho>0$ then $\mathcal{L}'(\rho)=0$ and $\rho^{(x)}(y)=\rho(y|x)$. 
\end{proposition}

\begin{itemize}
    \item As a direct consequence of Proposition \ref{prop: equivalent conditions}, 
    \item if $\alpha_{12}(x,y)=0$ then $A_{12}(x,y)=A_{21}(x,y)=0$; thus  $\mathcal{L}'_{xy}(\rho)=0$ for all $\rho$. In this case, Proposition \ref{prop: equivalent conditions} 
    \item states that $\mathcal{L}'\rho=0$ 
\end{itemize}

\begin{note}{\textup{Proposition \ref{prop: equivalent conditions} provides necessary and sufficient conditions when the measure $\rho^{(x)}$ and $\rho(\cdot|x)$ coincide. Below we provide some further comments on this proposition and on averaging of the system \eqref{eq:general SDE}.}}
\textup{
\begin{itemize}
    \item As a direct consequence of Proposition \ref{prop: equivalent conditions}, if $\alpha_{12}(x,y)=0$ then $A_{12}(x,y)=A_{21}(x,y)=0$; thus  $\mathcal{L}'_{xy}(\rho)=0$ for all $\rho$. In this case, Proposition \ref{prop: equivalent conditions} states that $\mathcal{L}'\rho=0$ and $\rho^{(x)}(y)=\rho(y|x)$ if and only if $\rho$ satisfies  $\mathcal{L}_x'\rho=0 ~\text{and}~
   \mathcal{L}'_y\rho=0$  (that is,  $\rho$ is in the kernel of both $\mathcal{L}_x'$ and $\mathcal{L}_y'$). Note that for the reversible system  \ref{eq:sfgradient} it is the case that $\cL_x'\rho_V=\cL_y'\rho_V=0$, so this is another way of seeing that for reversible systems equality \eqref{eqn:frozen=conditional} holds.
\item  If there is an explicit scale separation in \eqref{eq:general SDE}, that is if we consider the system
\begin{subequations}
%\label{eq:general SDE eps}
\begin{align*}   
dX_t&=f(X_t, Y_t)\,dt+\alpha_{11}(X_t,Y_t)dU_t+\alpha_{12}(X_t,Y_t)\, dW_t,\\
dY_t&=\frac{1}{\varepsilon}g(X_t, Y_t)\,dt+\frac{1}{\sqrt{\varepsilon}}\beta(X_t,Y_t)\, dW_t,
\end{align*}
\end{subequations}
then, using computations similar to \cite[Chapter 10]{pavliotis2014stochastic}, one obtains the following averaged dynamics:
\begin{equation*}
 d \bar X_t=\bar{b}_1(\bar X_t)\,dt+ \bar{\alpha}^{\frac{1}{2}}(\bar X_t)\,dU_t,  
\end{equation*}
where
\begin{align*}
\bar{f}(x)=\int f(x,y)\rho^{(x)}(y)\,dy,~\bar{\alpha}(x)=\int (\alpha_{11}\alpha_{11}^T+\alpha_{12}\alpha_{12}^T)(x,y)\rho^{(x)}(y)\,dy.
\end{align*}
Compared to \eqref{eqn:averaged coefficients}, the correlated noise gives rise to an additional term in the averaged diffusion coefficient. Note that the averaged dynamics and the one obtained by naively applying the PM (i.e. by forgetting for a moment that $\rho(\cdot \vert x)$ would depend on $\eps$) would coincide provided that $\rho^{(x)}(y)=\rho(y|x)$.
\end{itemize}
}  
\end{note}

\begin{proof}[Proof of Proposition \ref{prop: equivalent conditions}]
Let $\nu$ be any smooth probability density on $\R^n$. Using again the basic relation between the joint, conditional and marginal distributions \eqref{eqn:joint=conditional times marginal} (applied to $\nu$), we can rewrite $\mathcal{L}'_y\nu$ as follows:
\begin{align}
\label{eq:generator Y}
\mathcal{L}'_y\nu&= \mathcal{L}'_y[\nu(y|x)\bar\nu(x)]\notag\\
&=-\mathrm{div}_y[g(x,y)(\nu(y|x)\bar\nu(x))]+\mathrm{div}_y[\mathrm{div}_y(A_{22}\nabla_y(\nu(y|x)\bar\nu(x))]\notag
\\&=-\bar\nu(x)\mathrm{div}_y[g(x,y)\nu(y|x)]+\bar{\nu}(x)\mathrm{div}_y[\mathrm{div}_y(A_{22}\nabla_y(\nu(y|x)]\notag
\\&=\bar{\nu}(x)\mathcal{L}'_y(\nu(y|x)).
\end{align}
Substituting this back to \eqref{eq: full generator}, we obtain
\begin{equation}
\label{eq:invariant general SDE}
\mathcal{L}'\nu=\mathcal{L}_x'\nu+\mathcal{L}'_{xy}\nu+\bar{\nu}(x)\mathcal{L}'_y(\nu(y|x)).  \end{equation}
Now suppose that $\mathcal{L}'\rho=0$ and $\rho^{(x)}(y)=\rho(y|x)$. From the latter fact, we have
\[
\mathcal{L}'_y(\rho(y|x))=\mathcal{L}'_y(\rho^{(x)}(y))=0.
\]
Combining this, \eqref{eq:invariant general SDE} and the fact that $\mathcal{L}'\rho=0$ we obtain
\[
\mathcal{L}_x'\rho=-\mathcal{L}'_{xy}\rho.
\]
Thus if $\mathcal{L}'\rho=0$ and $\rho^{(x)}(y)=\rho(y|x)$ then we get
\begin{equation*}
%\label{eq:condition2}
\mathcal{L}_x'\rho=-\mathcal{L}'_{xy}\rho~\text{and}~ \mathcal{L}'_y\rho=0.   
\end{equation*}
Vice versa, suppose that the above relations hold. Then it is clear that
\[
\mathcal{L}'\rho=\mathcal{L}'_x\rho+\mathcal{L}'_{xy}\rho+\mathcal{L}_y'\rho=0.
\]
In addition, using \eqref{eq:generator Y} we obtain
\[
\bar{\rho}(x)\mathcal{L}'_y(\rho(y|x))=\mathcal{L}_y'\rho=0.
\]
As a consequence, if $\bar\rho(x)>0$ for all $x$, then we have $\mathcal{L}'_y(\rho(y|x))=0$, which implies that $\rho(y|x)=\rho^{(x)}(y)$ since $\rho^{(x)}(y)$ is the unique invariant measure of $\mathcal{L}_y$. This completes the proof of this proposition.
\end{proof}
We note that Proposition \ref{prop: equivalent conditions} is applicable to the gradient flow (reversible) system \eqref{eq:sfgradient} considered in Section \ref{sec: Main Results} in which the noises are independent and the invariant measure $\rho=\rho_V$ given explicitly in \eqref{eqn:gradient inv meas} is invariant to both the dynamics of $x_t$ and $y_t$ in \eqref{eq:sfxgradient} and \eqref{eq:sfygradient}. 

We also emphasize that the condition in Proposition \ref{prop: equivalent conditions} is not about the reversibility of the full dynamics. We next provide two examples of non-reversible systems; in the first  $\rho^{(x)}(y)=\rho(y|x)$, in the second $\rho^{(x)}(y)\neq\rho(y|x)$.

\begin{Example}[Example of a non-reversible system such that $\rho^{(x)}(y)=\rho(y|x)$]\label{example: rhox=rhoygivenx}
We consider the following system
\begin{equation}\label{non-reversible with J2}
    dZ_t = (J-I)\nabla V(Z_t) dt + \sqrt 2 dW_t
\end{equation}
for the variable $Z_t=(X_t, Y_t)\in \R^{d}\times \R^{(n-d)}$ where $J\in \R^{n\times n}$ is an anti-symmetric block matrix $J=\mathrm{diag}(J_1, J_2)$, with $J_1\in \R^{d\times d}$ and $J_2\in \R^{(n-d)\times (n-d)}$ are also anti-symmetric matrices. More precisely, the dynamics for $X_t$ and $Y_t$ are given by
\begin{subequations}
   \begin{align}
    dX_t&=(J_1-I)\nabla_x V(x,y)\,dt+\sqrt{2}dU_t,\label{X}\\
    dY_t&=(J_2-I)\nabla_y V(x,y)\,dt+\sqrt{2}dW_t.\label{Y}
\end{align}
\end{subequations}
Note that the noise terms are independent. Under suitable conditions on the potential $V$, \eqref{non-reversible with J2} is a non-reversible system and has a unique invariant measure $\rho$ with the same density as in the reversible system \eqref{eq:sfgradient}, which is given in \eqref{eqn:gradient inv meas}~\cite{hwang2005accelerating,lelievre2013optimal}. We note that the two dynamics \eqref{X} and \eqref{Y} have the same structure as in \eqref{non-reversible with J2}. Their adjoint generators are respectively given by
\begin{align*}
\mathcal L'_x \nu&= \mathrm{div}_x[(I-J_1)\nabla_{x} V(x,y)\nu]+\Delta_x \nu,
\\ \mathcal L'_y\nu &= \mathrm{div}_y[((I-J_2)\nabla_{y} V(x,y)\nu]+\Delta_y \nu. 
\end{align*}
We now show that $\mathcal L'_x\rho=\mathcal L'_y\rho=0$. In fact, we can rewrite $\mathcal L'_x$ as
\[
\mathcal L'_x \nu=\mathrm{div}_x\Big[\nu((I-J_1)\nabla_x V+\nabla_x \log \nu)\Big].
\]
Thus
\begin{align*}
\mathcal L'_x \rho&=\mathrm{div}_x\Big[\rho((I-J_1)\nabla_x V-\nabla_x V)\Big]=-\mathrm{div}_x(\rho J_1\nabla_x V)
\\&=-\Big(\rho \mathrm{div}_x(J_1\nabla_x V)+J_1\nabla_x V\cdot\nabla_x\rho\Big)    
\\&=-\rho\Big(\sum_{ij}(J_1)_{ij}\frac{\partial^2 V}{\partial x_i\partial x_j}-J_1\nabla_x V\cdot \nabla_x V\Big)=0,
\end{align*}
where the last equality holds because tyhe matrix $J_1$ is anti-symmetric.
Similarly we also get $\mathcal L'_y(e^{-V(x,y)})=0$. Furthermore, it is clear that $\bar{\rho}(x)=\int e^{-V(x,y)}\,dy>0$ for all $x\in \R^d$. Hence by Proposition \ref{prop: equivalent conditions}, $\rho^{(x)}(y)=\rho(y|x)$.
\end{Example}
\begin{Example}[non-reversible system such that $\rho^{(x)}(y)\neq \rho(y\vert x)$]\label{example: rhox not equal rhoygivenx}
Now we consider again the non-reversible system \eqref{non-reversible with J2}, but with $J$ being the classical symplectic block matrix 
\[
J=\begin{pmatrix}
0&I\\ -I&0    
\end{pmatrix}.
\]
In this case, the dynamics for the components $X_t$ and $Y_t$ are given by
%\todo{Need to say that $\rho$ is invariant?}
\begin{align*}
    dX_t&=\nabla_y V(x,y)\,dt-\nabla_x V(x,y)\,dt+\sqrt{2}dU_t,\\
    dY_t&=-\nabla_x V(x,y)\,dt-\nabla_y V(x,y)\,dt+\sqrt{2}dW_t.
\end{align*}

Under suitable conditions on the potential $V$, the above system also has a unique invariant measure $\rho$ with the same density given in \eqref{eqn:gradient inv meas}~\cite{hwang2005accelerating,lelievre2013optimal}. Note that one can write the dynamics for $Y$ as
\begin{equation}
\label{Y2}
dY_t=-\nabla_y (V(x,y)+ A(x,y))\,dt+\sqrt{2}dW(t),   
\end{equation}
if $A$ is a vector field satisfying
\[
\nabla_x V(x,y)=\nabla_y A(x,y).
\]
For simplicity we consider the one dimensional case, in which under suitable conditions on the potential $V$, such a vector field $A$ always exists; for instance, one can take
\[
A(x,y)=\int_0^y \partial_y V(x,w)\,dw.
\]
We can view \eqref{Y2} as a gradient system with the external potential $V+A$.  It implies that
\[
\rho^{(x)}(y)= \frac{e^{-(V(x,y)+ A(x,y))}}{\int e^{-(V(x,y)+ A(x,y))}\,dy}.
\]
We recall that $\rho(y|x)=\frac{e^{-V(x,y)}}{\int e^{-V(x,y)}\,dy}$. Thus $\rho^{(x)}(y)=\rho(y|x)\Longleftrightarrow   e^{-A} = \frac{\int e^{-V-A} dy}{\int e^{-V}dy}$. In general, this condition is not satisfied, thus $\rho^{(x)}(y)\neq\rho(y|x)$. Alternatively, one can directly check that 
\[
\mathcal{L}'_x\rho=-\mathrm{div}_x(\rho\nabla_y V),\quad \mathcal{L}'_y\rho=-\mathrm{div}_y(\rho\nabla_x V). 
\]
In general, it does not hold that $\cL_x'\rho = \cL_y'\rho = 0$; thus according to Proposition \ref{prop: equivalent conditions} $\rho^{(x)}\neq \rho(y|x)$.
\end{Example}    

%%%%%%%%%%%%%%%%5
%%%%%%%%%%%%%%%%%5
%%%%%%%%%%%%%%%%%%%%5
%%%%%%%%%%%%%%%%%%%%%%55

\subsection{Formal asymptotic  expansion of the ECD}\label{sec: Carsten's section}

We now discuss a similar question as in the previous paragraph, but we consider the case of slow-fast systems (\ref{eq:sf}). 
We recall that in this case the invariant measure of the system, $\rho^{\eps}(x,y)$,  can depend on $\eps$, and this also holds for the conditional distribution $\rho^\eps(y|x)$, 
\begin{equation}
\label{eq:conditional density}
    \rho^\eps(y|x) = \frac{\rho^\eps(x,y)}{\bar{\rho}^\eps(x)}\,,\quad \bar{\rho}^\eps(x)=\int\rho^\eps(x,y)\,dy \, ,
\end{equation}
As explained in Section \ref{sec: Main Results}, here we discuss the question of when the limit \eqref{limit scale separation} holds, i.e. the question of when averaging and projection method are consistent in the limit of scale separation. 

 Before moving on with the formal asymptotic expansion, we illustrate with an example what can go wrong. 
\begin{Example}\label{example:short version cont'd}

We consider the situation of Example \ref{example:short version}. The solution  $(X^\eps,Y^\eps)$ to (\ref{eq:sfeps}) 
% \begin{subequations}\label{eq:sfeps1}
% \begin{align}
%         dX^\eps_t &= -\left(X^\eps_t - Y^\eps_t\right) dt \,,\quad X^\eps_0=x\label{example:Carsten'sexample1eps2}\\
% dY^\eps_t &= -\frac{1}{\eps}Y^\eps_t \,  dt + \sqrt{\frac{2}{\eps}} \, dW_t \,,\quad Y^\eps_0=y\label{example:Carsten'sexample2eps2} \,,
%     \end{align}
%   \end{subequations}
has a unique Gaussian invariant measure $\rho^\eps$ given by \eqref{Gaussian invariant}.
%\[
%\rho^\eps=\cN(0,\Sigma^\eps)\,,\quad \Sigma^\eps=\begin{pmatrix}
%    \frac{\eps}{1+\eps} & \frac{\eps}{1+\eps} \\ \frac{\eps}{1+\eps} & 1
%\end{pmatrix}\,.
%\]
%{\color{red}M: refer to formulas in Section 2 without rewriting }
Clearly, $\rho^\eps>0$ for every $\eps>0$, but the measure becomes singular in the limit $\eps\to 0$; in particular $\bar{\rho}^\eps\stackrel{*}{\rightharpoonup}\delta_0$ ceases to be strictly positive for $\eps=0$. As a consequence, the conditional distribution $\rho^\eps(\cdot|x)$ as given by   \eqref{eq: ECD} does 
%\[
%\rho^\eps(\cdot|x) = \cN\left(x,1-\frac{\eps}{1+\eps}\right)
%\]
not converge to $\rho^{(x)}=\cN(0,1)$ as $\eps\to 0$.  
\end{Example}

This example is consistent with what we find below via asymptotic expansion: a crucial assumption for the limit \eqref{limit scale separation} to hold is that the marginal $\bar{\rho}^{\eps}$ of the invariant measure $\rho^{\eps}$ of the system remains a strictly positive function in the limit $\eps \rightarrow 0$. We will also identify the limit of $\bar{\rho}^{\eps}$ as the stationary solution of the averaged problem, see \eqref{eq:solvability}  below. 

The Fokker-Planck operator associated to \eqref{eq:sf} is given by
$$
(\cL^{\eps})'= \cL_x'+\frac{1}{\eps} \cL_y'\,,
$$
where we recall that $\cL_x'$ and $\cL_y'$ have been defined in \eqref{flat adjoint}. The invariant measure of \eqref{eq:sf} is then the solution (which we assume to be unique) of

\begin{equation}\label{eq:FP0}
    (\cL^{\eps})'\rho^\eps = 0\, \,.
\end{equation}

We write $\rho^\eps(x,y) = \bar{\rho}^\eps(x)\rho^\eps(y|x)$ and, as customary \cite{pavliotis2008multiscale}, make the ansatz
\begin{subequations}\label{ansatz}
\begin{alignat}{1}
    \bar{\rho}^\eps(x) & = \bar{\rho}_0(x) + \eps\bar{\rho}_1(x) + \ldots  \label{ansatz1}\\
    \rho^\eps(y|x) & = \rho_0(x,y) + \eps\rho_1(x,y) + \ldots \,\,.\label{ansatz2}
\end{alignat}    
\end{subequations}
We also make the key assumption that the function $\bar\rho_0$ is smooth and  strictly positive for every $x \in \R^d$, $\bar \rho_0>0$.  

Plugging \eqref{ansatz}  into (\ref{eq:FP0}) and equating different powers of $\eps$, we obtain a hierarchy of equations; in particular, by equating the terms of order $\eps^{-1}$ and the order one terms, we obtain
\begin{subequations}\label{eq:epsis}
    \begin{alignat}{2}\label{eq:epsi-1}
    0 & = - \cL_y' (\bar{\rho}_0(x) \rho_0(x,y))\\\label{eq:epsi0}
    \cL_x'  (\bar{\rho}_0(x) \rho_0(x,y)) & = - \cL_y' ( \bar{\rho}_0(x) \rho_1(x,y) + \bar{\rho}_1(x) \rho_0(x,y) ) \,.
    \end{alignat}
\end{subequations}
Since $\bar{\rho}_0>0$ is independent of $y$ and the kernel of $\cL_y'$ is the one-dimensional subspace of $L^2(dy)$ spanned by $\rho^{(x)}$, equation (\ref{eq:epsi-1}) implies 
$$
\rho_0(x,y)=h(x)\rho^{(x)}(y)+ \tilde{h}(x)
$$
for some functions $h, \tilde h$ which depend on the $x$ variable only. Imposing further the requirement that the integral of $\rho_0$ in the $y$ variable should be equal to one, for every $x \in \R^d$ (which is a natural requirement  in view of  the fact that $\rho^\eps(y\vert x)$ enjoys the same property, for every $\eps>0$ and every $x\in \R^d$), we conclude 
\begin{equation}
    \rho_0(x,y)=\rho^{(x)}(y)\,.
\end{equation} 
The above now implies two facts. First, equation \eqref{ansatz2} becomes
\begin{equation}\label{eqn:expansion2}
\rho^\eps(y|x) = \rho^{(x)}(y) + \eps\rho_1(x,y) + \ldots \,\,.
\end{equation}
Secondly,  equation (\ref{eq:epsi0}) simplifies to  
\begin{equation}
\label{eq:rho1}
     -   \cL_y' (\bar{\rho}_0(x)\rho_1(x,y))=\cL_x'  (\bar{\rho}_0(x) \rho^{(x)}(y)) \,.
\end{equation}
We view the above as a Poisson equation for the operator $\cL_y'$ in the unknown $\rho_1$. Integrating both sides of \eqref{eq:rho1} in the $y$ variable one can see that a necessary condition for  \eqref{eq:rho1} to admit a solution is that the right hand side averages to zero (in $y$). This yields, as a by-product, a characterisation of $\bar{\rho}_0$ as the equilibrium distribution with respect to the averaged dynamics,
since 
\begin{equation}\label{eq:solvability}
    \cL'_A\, \bar{\rho}_0(x)  = \int  \cL_x' (\bar{\rho}_0(x) \rho^{(x)}(y))\,dy = 0\,.
\end{equation}

% with infinitesimal generator
% \begin{equation}
%     \bar{\cL}_x = \frac{1}{2}\bar{\alpha}\bar{\alpha}^T\colon\nabla_x^2 + \bar{f}\cdot\nabla_x\,,
% \end{equation}
% where 
% \begin{equation}
%     \bar{f}(x) = \int f(x,y)\rho^x(y)\,dy
% \end{equation}
% and 
% \begin{equation}
%         (\bar{\alpha}\bar{\alpha}^T)(x) = \int  (\alpha \alpha^T)(x,y)\rho^x(y)\,dy
% \end{equation}
% are the averaged drift and diffusion coefficients. 

%
Formally, if $\rho_1(x,y)$ is a well-behaved function, then   \eqref{eqn:expansion2} states that $\rho^\eps(x,y)\to \rho^{(x)}(y)$ as $\eps\to 0$.    The properties of the function $\rho_1(x,y)$ can be deduced from the fact that it is a solution of the Poisson equation \eqref{eq:rho1}, which are well studied in the literature, see e.g. \cite{pardoux2003poisson, crisan2022poisson, weinan2005analysis, pavliotis2008multiscale} (without claim to completeness of references) or also \cite{stoltz2018longtime}, which deals with the rigorous proof of related asymptotic problems. 

 To turn the above formal expansion into an actual proof one needs to combine these results on Poisson equations together with a more careful study of the difference $\rho^{\eps}-\rho^{(x)}(y) - \eps \rho_1(x,y)$, which needs to be shown to converge to zero as $\eps \rightarrow 0$. While not too difficult, this is not within the scope of this paper.  

\begin{Example}\label{example:cond2aux}
We consider the following slight modification of of Example \ref{example:short version cont'd}:
\begin{subequations}\label{eq:sfeps2}
\begin{align}
        dX^\eps_t &= -\left(X^\eps_t - Y^\eps_t\right) dt  + \sqrt{2}dU_t\,,\quad X^\eps_0=x\\
dY^\eps_t &= -\frac{1}{\eps}Y^\eps_t \,  dt + \sqrt{\frac{2}{\eps}} \, dW_t \,,\quad Y^\eps_0=y \,.
    \end{align}
  \end{subequations}
This dynamics has a Gaussian invariant measure 
\[
\rho^\eps=\cN(0,\Sigma^\eps)\,,\quad \Sigma^\eps=\begin{pmatrix}
    \frac{1+2\eps}{1+\eps} & \frac{\eps}{1+\eps} \\ \frac{\eps}{1+\eps} & 1
\end{pmatrix}\,,
\]
with strictly positive limit density $\cN(0,I)$. The corresponding marginal and conditional distributions, 
\[
\bar{\rho}^\eps = \cN\left(0,\frac{1+2\eps}{1+\eps}\right)
\]
and 
\[
\rho^\eps(\cdot|x) = \cN\left(\frac{\eps x}{1+2\eps},1-\frac{\eps^2}{1+3\eps+4\eps^2}\right), 
\]
converge to the correct limit distributions $\bar{\rho}=\cN(0,1)$ and $\rho^{(x)}=\cN(0,1)$ as $\eps\to 0$. The projected equation for \emph{finite} $\eps$ is given by 
\begin{equation}
        dX^\eps_P(t) = -X^\eps_P(t)\left(1 - \frac{\eps}{1+2\eps}\right) dt  + \sqrt{2}dU(t)\,,\quad X^\eps_P(0)=x
\end{equation}
and it follows by It\^o's formula and a standard Gronwall estimate that, as $\eps\to 0$, the projected dynamics converges pathwise to the solution of the averaged equation,
\begin{equation}
        dX_A(t) = -X_A(t)dt  + \sqrt{2}dU(t)\,,\quad X_A(0)=x\,,
\end{equation}
on any finite time-interval\,.
\end{Example}
% \textcolor{teal}{C: I am not sure whether this is enough, but I did not manage to prove anything stronger than pointwise convergence (or uniform on compact subsets, in contrast to what I had in mind), but I also couldn't find anything beyond general statements like in \cite[Secs. 2.3 and 3.2]{weinan2005analysis} that exponential mixing and smooth coefficients imply boundedness.}
% \textcolor{red}{HD: how this statement follows from Theorem B.2 is not clear to me. could you explain a bit more?}

\section{Statement and Proof of Theorem \ref{thm:projection method works}}\label{sec: statement and proof of Thm 6.2}

\subsection{Non-autonomous SDEs}
\label{subsec: non-autonomous SDEs main text} We have so far only informally stated the results of \cite{angiuli2013hypercontractivity, cass2021long}. Here we extend them slightly, see Theorem \ref{thm:appendixnonautonomous} below, which we state here and prove in Appendix \ref{appendix:long time behaviour of non-autonomous SDEs}. The extension is minimal, and only needed to fit our framework, see Note \ref{noteappendix:note on assumption non autonomous SDEs}.

\begin{assumption}\label{ass:Lunardiappendix} 
The coefficients $B, \sigma$ of \eqref{stratonovich non-autonomous general SDE} are smooth in both arguments and the diffusion coefficient $\sigma$ is globally Lipshitz. Moreover:
		\begin{enumerate}[label=\textnormal{[C\arabic*]},ref={[C\arabic*]}]
		\item\label{ass:unformellipticity appendix} The diffusion coefficient $\sigma$ is uniformly elliptic, i.e. there exists a constant 
  $\eta>0$, independent of $(t,x) \in \R_+ \times \R^d$, such that
  $$
 \eta |\xi|^2 \leq \langle\sigma(t,x) \xi, \xi \rangle, \quad \mbox{for every } \xi \in \R^d \,.
  $$ 
  \item\label{ass:Lyapcondappendix} There exist constants $c_1\geq 0, c_2>0$ and a function $\varphi\in C^2(\R^d)$ with compact level sets such that 
  $$
  \cL_t \varphi (x) \leq c_1-c_2 \varphi(x) \quad \mbox{for every } x \in \R^d, 
  $$
  where $\cL_t$ is the generator of the dynamics \eqref{stratonovich non-autonomous general SDE}, namely
  $$
  \cL_t \psi(x)= B(t,x) \cdot \nabla_x\psi(x)+\sum_{i=1}^d\sigma\col(t,x) \cdot\nabla_x(\sigma\col(t,x)\cdot \nabla_x\psi (x)) \,.
  $$
\item\label{ass:ObtuseAngleConditionappendix} There exists  a constant $\lambda_0>0$ such that the following holds for every function $h\colon\R^d\rightarrow \R$ sufficiently smooth,   for every $i=1, \dots, d$, and for every $x \in \R^d$:
$$
\left([\sigma\col \cdot \nabla_x, B\cdot \nabla_x] h(x)- 
(\pa_t\sigma\col )\cdot\nabla_x h(x) \right) \left( \sigma\col \cdot \nabla_x h(x)\right)
\leq - \lambda_0 \left\vert \sigma\col\cdot \nabla_xh(x)\right\vert^2\,,  
$$
where the brackets $[\cdot, \cdot]$ in the above denote the commutator between two differential operators. More explicitly:
$$
[\sigma\col \cdot \nabla_x, B\cdot \nabla_x] h(x) = 
\sum_{j=1}^d \sigma_{ji}(t,x)\pa_{x_j}(B^j(t,x)\pa_{x_j}h) - 
B^j(t,x)\pa_{x_j}(\sigma_{ji}(t,x)\pa_{x_j}h) \,.
$$
\item\label{ass:coefficients have limit appendix}
There exist smooth functions $\bar B, \bar\sigma\col\colon\R^d\rightarrow\R^d$, such that
$$
\lim_{t\rightarrow \infty} B(t,x) = \bar B(x), \quad \lim_{t\rightarrow \infty}\sigma\col(t,x)=\bar\sigma\col(x), \quad \mbox{for every } x \in \R^d, i=1, \dots, d. 
$$
   \end{enumerate}  
   \end{assumption}   
In what follows we denote by $\mathcal X_{s,x}(t)$ the solution  of \eqref{stratonovich non-autonomous general SDE} at time $t$,  when the dynamics was started at time $s\geq 0$    from the initial condition $x \in \R^d$, i.e. $\cX_{s,x}(s)=x$.  

\begin{theorem}\label{thm:appendixnonautonomous}
Let Assumption \ref{ass:Lunardiappendix} hold. Then:

a) There exists a probability measure $\mu$ on $\R^d$ such that for every function $h\colon\R^d\rightarrow \R$ which is continuous and bounded the following limit holds: 
$$
    \lim_{t\rightarrow \infty} \mathbb E[h(\cX_{s,x}(t))] = \int_{\R^d} h(x) \mu(dx) \,, \qquad \mbox{for every } s\geq0, x\in\R^d \,.
$$
b) The dynamics \eqref{stratonovich autonomous limiting SDE} is well posed and it 
admits a unique invariant measure, which is precisely the measure $\mu$. Furthermore the process \eqref{stratonovich autonomous limiting SDE} converges weakly to $\mu$, that is, 
$$
\lim_{t\rightarrow \infty} \mathbb E [h(\bar{\mathcal X}_t) ] = \int_{\R^d} h(x) \mu(x) dx, 
$$
for every $h\colon\R^d\rightarrow \R$ which is continuous and bounded (and irrespective of the initial datum $\bar{\mathcal X}_0$). 
\end{theorem}
The proof of Theorem \ref{thm:appendixnonautonomous} 
%as well as comments to explain the meaning of   each of the conditions of Assumption \ref{ass:Lunardiappendix} 
can be found in Appendix \ref{appendix:long time behaviour of non-autonomous SDEs}. In the note below we just comment on the assumptions of Theorem \ref{thm:appendixnonautonomous}. 
\begin{note}\label{noteappendix:note on assumption non autonomous SDEs}{\textup{Theorem \ref{thm:appendixnonautonomous} is a slight extension of the results of \cite[Section 6]{angiuli2013hypercontractivity} to the case when the diffusion coefficient is globally Lipshitz rather than bounded and it is a function of both time and space rather than a function of time only. This small extension requires no new ideas and can be achieved leveraging on \cite[Section 6]{cass2021long}.  \\
The non-autonomous evolution \eqref{stratonovich non-autonomous general SDE} gives rise to a two-parameter semigroup $\cQ_{s,t}$, namely
$$
(\cQ_{s,t}h)(x):= \mathbb E h(\cX_{s,x}(t)), \quad h \in C_b(\R^d) \,.
$$
When working with non-autonomous SDEs the concept of invariant measure is replaced by the concept of {\em evolution system of measures} (ESM), see e.g. \cite{da2007ornstein, angiuli2013hypercontractivity}. A family $\{\mu_t\}_{t\geq 0}$ is an ESM for the two-parameter semigroup $\cQ_{s,t}$ if 
$$
\int_{\R^d} \cQ_{st}h (x) \mu_s(dx) = \int_{\R^d} h(x) \mu_t(dx), 
$$
for every $0\leq s\leq t$ and $h \in C_b(\R^d)$. Assumption \ref{ass:Lyapcondappendix} ensures the existence of an evolution system of measures $\{\mu_t\}_{t\geq 0}$ and it moreover implies that such a system is tight. Assumption \ref{ass:ObtuseAngleConditionappendix} then implies uniqueness of the ESM. Since such an ESM is tight, one can take (up to subsequences) the limit of $\mu_t$, which is then the natural candidate for the limiting behaviour of the semigroup $Q_{st}$. Finally, letting $t$ to infinity in the inequality of Assumption \ref{ass:Lyapcondappendix} one can see that such an assumption implies that the function $\varphi$ is a Lyapunov function also for the generator of \eqref{stratonovich autonomous limiting SDE}, hence the existence of an invariant measure for \eqref{stratonovich autonomous limiting SDE}. Such an invariant measure is unique because our uniform ellipticity assumption \ref{ass:unformellipticity appendix} implies the uniform ellipticity of $\bar \sigma$. The proof of the theorem is then concluded by showing that $\mu$ coincides with the invariant measure of \eqref{stratonovich autonomous limiting SDE}.\\
The fact that Assumption \ref{ass:ObtuseAngleConditionappendix} implies uniqueness of the ESM is perhaps the least trivial statement in this discussion, so let us comment a little more on this.  First of all note that if $\sigma=\sigma(t)$ and it is bounded  (which is the setting of \cite{angiuli2013hypercontractivity}) then   Assumption \ref{ass:ObtuseAngleConditionappendix}  boils down to \cite[Condition (2.3)]{angiuli2013hypercontractivity} (this is perhaps easier to see in the simpler case when $d=1$ and $\sigma$ is constant). Assumption \ref{ass:ObtuseAngleConditionappendix}  is a specific form of a condition which was termed in \cite{crisan2016pointwise} the {\em obtuse angle condition} (OAC). The OAC implies that the space-derivatives of the semigroup associated to the dynamics decay exponentially fast in time. If the semigroup is a one-paramter semigroup (i.e. associated to an autonomous evolution) then this decay of the derivatives has been shown to imply uniqueness of the invariant measure \cite{cass2021long, crisan2021uniform}. This can be seen as a simple consequence of the fundamental theorem of calculus, see e.g. \cite{cass2021long}.  It is then reasonable that, when dealing with two-parameter semigroups, this decay of derivatives implies uniqueness of  the evolution system of measures, as shown in the works \cite{kunze2010nonautonomous, angiuli2013hypercontractivity} and references therein. 
    }}
\end{note}

To demonstrate in practice what assumption \ref{ass:ObtuseAngleConditionappendix} means, consider the simple one-dimensional  SDE
$$
dX_t = -f(t) X_t dt + \sigma dW_t \, 
$$
where $f$ is a real valued function, assumed to be good enough that the above SDE is well posed and $\sigma$ is a strictly positive constant. Then, in the notation of the above theorem, we have $d=1$, $B(t,x)= -f(t) x$ and $\sigma^{[1]}=\sigma$. Hence Assumption \ref{ass:ObtuseAngleConditionappendix} is satisfied if the inequality 
$$
\left([\sigma \partial_x, -f(t) x \partial_x] h(x)\right) (\sigma \partial_x h) \leq -\lambda_0 |\sigma \partial_x h (x)|^2, \quad x \in \R\, ,
$$
holds for some $\lambda_0>0$ and  for every sufficiently smooth function $h$. After calculating the bracket in the above, this inequality boils down to 
$$
f(t)\left\vert \sigma \pa_x h \right \vert^2 \geq \lambda_0 \left\vert \sigma \pa_x h \right \vert^2 \,, 
$$
which, as expected, is satisfied if  $f(t)\geq \lambda_0$.

\subsection{Statement and proof of Theorem \ref{thm:projection method works}}

We are finally in a position to state the main theorem of this section.
\begin{theorem}\label{thm:projection method works} Suppose the following three assumptions hold:

i) The SDE \eqref{eq:sfep=1} has smooth coefficients, it is uniformly elliptic and it admits a unique weak solution. Moreover such an SDE is ergodic,  in the sense that it admits a unique invariant measure, $\rho$,  and satisfies assumption ii)  of Proposition \ref{Prop: extension gyongy}. 

ii) The coefficients of the  SDE \eqref{eqn:gyongy-coarse-graining} are smooth and the SDE admits a unique weak solution (and the law of  such a solution has a density which is absolutely continuous with respect to the Lebesgue measure). 

iii) The coefficients of   the non-autonomous SDE  \eqref{eqn:gyongy-coarse-graining} satisfy Assumption \ref{ass:Lunardiappendix} of Theorem \ref{thm:appendixnonautonomous}. 

Then the dynamics \eqref{eq:sfgenaralprojection} admits a unique  invariant measure and such an invariant measure coincides precisely with the marginal   $\bar\rho$ (defined in \eqref{eqn:bar rho def}) of the invariant measure $\rho$ of \eqref{eq:sfep=1}. Moreover,  
    $$
    \lim_{t\rightarrow \infty} \mathbb E ( h(X_P(t)) )= \int_{\R^d} h(x) \bar{\rho}(x) dx\, ,
    $$
    for every $h \in C_b(\R^d)$. 
\end{theorem}

\begin{proof}[Proof of Theorem \ref{thm:projection method works}] It is a straightforward consequence of Proposition \ref{Prop: extension gyongy} and Theorem \ref{thm:appendixnonautonomous}, using the reasoning outlined in Section \ref{sec:heuristic explanation of theorem}, which we do not repeat here. 
Assumptions i) and ii) imply that     Proposition \ref{Prop: extension gyongy} can be applied. Indeed, if the coefficients of \eqref{eq:sfep=1} are smooth and the SDE is uniformly elliptic then the density  $\rho_t= \rho_t(x,y)$ is strictly positive for every $t\geq 0$. Hence the conditional density $\rho_t(y\vert x)$ (defined as in \eqref{joint and marginal time t}) is strictly positive as well. Moreover, because of the form of the overall diffusion matrix of \eqref{eq:sfep=1}, the uniform ellipticity of \eqref{eq:sfep=1} implies that the matrix $\alpha$ in \eqref{eq:sfxep=1} is itself uniformly elliptic. Using \eqref{eqn:conditional is a prob measure }, a straightforward calculation then shows that also the diffusion matrix $\alpha_G$ is uniformly elliptic, hence the law of the solution of \eqref{eqn:gyongy-coarse-graining} admits again a strictly positive density for every time $t\geq 0$.  Assumption iii) clearly guarantees that we can use 
Theorem \ref{thm:appendixnonautonomous} as well, with $X_G(t)$ in place of \eqref{generalSDEnon-autonomous} and $X_P(t)$ in place of \eqref{stratonovich autonomous limiting SDE}. This concludes the proof.
\end{proof}

\begin{note}\label{note:on theor3_5}{\textup{
Let us make some comments on the assumptions of Theorem \ref{thm:projection method works}.  As it is clear from the proof,  Assumptions i) and ii) are needed to make sure that \eqref{eqn:law equality gyongy} holds, i.e., roughly speaking,  that we can apply the coarse graining method proposed by  G\"yongy. Assumption iii) is instead needed to apply the results on non-autonomous SDEs that we have already mentioned.
\begin{itemize}
\item The reasoning in the proof of Theorem \ref{thm:projection method works} implies that, since \eqref{eq:sfep=1} is unformly elliptic, then $\alpha_G$ is a uniformly elliptic matrix as well. Hence condition \ref{ass:unformellipticity appendix} of Assumption \ref{ass:Lunardiappendix} is automatically satisfied. 
    \item The coefficients of equation \eqref{eqn:gyongy-coarse-graining} are only known implicitly, in the sense that while $f,g,\alpha, \beta$ are known, the conditional density $\rho_t(y\vert x)$ may be not explicitly available, so that checking the validity of Assumption iii) might be difficult in practice, and we would like to be upfront about this. So, similarly to the remarks we have made in Note \ref{note: on gyongy extension},  more work is required to provide the user with conditions on $f, g, \alpha, \beta$ which are sufficient to guarantee that Assumption iii) holds.  This is not straightforward and will be the object of future work. Nonetheless here we observe that, with a calculation completely analogous to those at the end of the previous subsection, assuming system \eqref{eq:sfep=1} is two dimensional (i.e. $d=n-d=1$) with $\alpha(x,y)=\sigma$, then  Assumption \ref{ass:ObtuseAngleConditionappendix} is satisfied by \eqref{eqn:gyongy-coarse-graining} if $\pa_x f_G(t,x) \leq -\lambda_0$, for every $t,x$ and for some $\lambda_0>0$. 
    \item In our heuristic explanation of Theorem \ref{thm:projection method works} we mentioned the convergence of $\rho_t(y \vert x)$ to $\rho(y\vert x)$, as $t\rightarrow \infty$. This condition is slightly hidden in  Theorem \ref{thm:projection method works}, so we emphasize that it is contained in Assumption iii), more precisely in Assumption \ref{ass:Lunardiappendix}, condition \ref{ass:coefficients have limit appendix}. Indeed, assuming for example that $B$ is polynomial of degree $q$, if  $\rho_t(y\vert x)$ converges to $\rho(y\vert x)$ in $q$-Wasserstein distance (for every $x$ fixed), then Assumption \ref{ass:Lunardiappendix}, condition \ref{ass:coefficients have limit appendix} holds. 
\end{itemize}
}
} 
\end{note}

\subsection{Example where the Projection Method fails}\label{sec: counterexample to Gyongy}

Let us consider again the system in Example \ref{example:short version}, this time setting  $\varepsilon=1$. That is, consider the process
\begin{align}
dX_t &= (-X_t+Y_t) dt \label{example:Carsten'sexample1}\\
dY_t &= -Y_t \,  dt + \sqrt{2} \, dW_t \label{example:Carsten'sexample2} \,.
\end{align}
 and let $A$ and $C$ be the matrices
\begin{equation}\label{eq:counterexampleCoeff}
A= 
\left( \begin{array}{cc}
-1 & 1 \\
0 & -1
\end{array}
\right), \quad C = 
\left( \begin{array}{cc}
0 \\
\sqrt{2}
\end{array}
\right) \,. 
\end{equation}
The solution $Z_t=(X_t, Y_t) \in \R^2$ of \eqref{example:Carsten'sexample1}-\eqref{example:Carsten'sexample2} can be written as 
 \begin{equation}\label{eqn:Carstenexample-solution}
     Z_t= \exp(At)Z_0 + \int_0^t \, 
     \exp(A(t-s))C\,d \bf{W}_s \,, 
 \end{equation}
 where ${\bf{W}}_t= (0, W_t)^T$. 
 We will carry out explicit calculations  below, but before doing any calculations we explain what goes wrong in this example. The dynamics \eqref{example:Carsten'sexample1}-\eqref{example:Carsten'sexample2}  is not uniformly elliptic, but it is hypoelliptic and it admits a unique invariant measure $\rho$. Uniqueness is a consequence of the fact that for Ornstein-Uhlenbeck processes hypoellipticity implies uniqueness of the invariant measure. See Lemma \ref{lemma B3} for further properties of this dynamics.  Moreover, again due to hypoellipticity,  even if the initial datum is deterministic, for any $t>0$ the law $\rho_t$ of the process has a density with respect to the Lebesgue measure (on $\R^2$) and the same is true of the invariant measure. So, to summarise, both $\rho_t$ and $\rho$ have a strictly positive smooth density, for every $t>0$. This implies that the marginals $\bar\rho_t$ and $\bar\rho$ have a strictly positive density too. 
 
The definition of the conditional distribution \eqref{joint and marginal time t} also implies that both $\rho_t(\cdot\vert x)$ and $\rho(\cdot\vert x)$ have strictly positive densities on $\R$ for every $x\in\R$ and $t>0$. However, since the diffusion coefficient $\alpha$ in \eqref{example:Carsten'sexample1} is zero, also the diffusion coefficients $\alpha_P$ and $\alpha_G$ of the resulting projection and Gy\"ongy coarse grained dynamics are zero (as they are obtained through \eqref{eqn:projection-coefficients2} and \eqref{eqn:gyongy-averaged-coefficients2}, respectively). That is, the approximations obtained through either one of these methods are deterministic ODEs. Therefore, if the initial datum is deterministic,  the law of the random variables $\xp(t)$ and $\xg(t)$ is a Dirac for any $t>0$ and therefore the law of $\xg(t)$ cannot coincide with the law of $X_t$ and the invariant measure of $\xp$ cannot  be $\bar \rho$. In our example this  compounds with the fact that the equation obtained via GM is not well posed for small $t$, in case of deterministic initial data. 

We make more precise observations below, but this discussion suffices to illustrate the fact that ellipticity of the SDE \eqref{eq:sfxep=1}-\eqref{eq:sfyep=1} is an important requirement in order for the Gy\"ongy and projection methods to produce the desired results.

Let us now give more detail. The unique invariant measure of (\ref{eqn:Carstenexample-solution}) is a mean zero Gaussian, $\rho \sim \mathcal N(0, \Sigma)$, where $\Sigma$ is the unique symmetric positive definite solution of the Lyapunov equation
$A\Sigma + \Sigma A^T = -CC^T$, namely 
\begin{equation}
\Sigma= \frac{1}{2}
\left( \begin{array}{cc}
1 & 1 \\
1 & 2
\end{array}
\right) \,.
\end{equation}

We suppose that $Z_0\sim\cN(m_0,\Sigma_0)$ where the special case of Dirac initial data $Z_0=z$ is included by the choice $m_0=z$ and $\Sigma_0=0$. To simplify calculations a little bit, we take $X_0$   independent of  $Y_0$ and $Z_0$ independent of the Brownian motion $W_t$. 
Then the law $\rho_t$ of $Z_t$ is also Gaussian, $\rho_t \sim \mathcal N(m_t, \Sigma_t)$ where $m_t= \mathbb{E}(Z_t)$, $\Sigma_t:=\mathrm{Cov}(Z_t) = \mathbb E (Z_tZ_t^T) - \mathbb E Z_t \mathbb E Z_t^T$, and the expressions for the mean and the covariance are 
\begin{equation}\label{eq:counterexampleMean}
        m_t  = \begin{pmatrix}
            e^{-t} m_0^x+te^{-t}m^y_0\\ e^{-t}m^y_0
        \end{pmatrix}
\end{equation}
and 
\begin{equation}\label{eq:counterexampleCov}
        \Sigma_t = \frac{1}{2}\begin{pmatrix}
            \displaystyle 1 - e^{-2t}\left(2t^2+2t+1 - 2\Sigma^{xx}_0 - 2t^2\Sigma^{yy}_0\right) & 1 - e^{-2t}\left(2t+1 - 2t \Sigma^{yy}_0\right)\\
           1 - e^{-2t}\left(2t+1 - 2t \Sigma^{yy}_0\right) & 2 - 2e^{-2t}\left(1 - \Sigma_0^{yy}\right) \,.
        \end{pmatrix}
\end{equation}
%{\color{red}  M: in the above the entry $\Sigma_t^{yy}$ should  be $\Sigma_t^{yy} = 2 - 2 e^{-2t}\left(1 - \Sigma_0^{yy}\right)$ I think. Could you please check if you agree and if this is coherent with the code? Might have been just a typo}

In  the above  we have used the standard notation
\[
m_t =\begin{pmatrix}
            m^x_t\\ m^y_t
        \end{pmatrix}\,,\quad \Sigma_t = \begin{pmatrix}
    \Sigma^{xx}_t & \Sigma^{xy}_t\\ \Sigma^{yx}_t & \Sigma^{yy}_t
\end{pmatrix},\quad t\ge 0\,
\]
and the expression \eqref{eq:counterexampleCov} has been obtained by direct computation, using the fact that \eqref{eqn:Carstenexample-solution} implies the following expression for the covariance matrix 
\begin{equation}
\Sigma_t =  \int_0^t ds \, e^{A(t-s)}CC^T e^{A^T(t-s)} + e^{At} \mathrm{Cov} (Z_0)  e^{A^Tt} \,, 
\end{equation}
with 
\[
\exp(At) = \begin{pmatrix} e^{-t} & t e^{-t} \\ 0 & e^{-t} \end{pmatrix}.
\]
% Upon choosing $Z_0= (X_0,Y_0)^T$ independent of the Brownian motion $W_t$,  we have  $m_t = \mathbb E(Z_t) = \mathbb E (e^{At}Z_0)$  and 
% \begin{equation}
% \Sigma_t = \mathbb E \left( \int_0^t ds \, e^{A(t-s)}BB^T e^{A^T(t-s)}\right) + e^{At} (\mathbb E Z_0Z_0^T)  e^{A^Tt} \,.
% \end{equation}
% Since 
% \begin{equation}
% e^{As}=
% \left( \begin{array}{cc}
% e^{-s} & se^{-s} \\
% 0 & e^{-s}
% \end{array}
% \right) \,, 
% \end{equation}
% we obtain\begin{align}
% \Sigma_t& =
% \left( \begin{array}{cc}
% \int_0^t 2s^2e^{-2s} & \int_0^t 2se^{-2s} \\
% \int_0^t 2se^{-2s} &  \int_0^t 2e^{-2s}
% \end{array}
% \right)  + e^{At} (\mathbb E Z_0Z_0^T)  e^{A^Tt} \\
% &=  
% \left( \begin{array}{cc}
% \frac12 \left(e^{-2t} (-2t^2-2t-1)+1 \right) &  \frac12 (1-e^{-2t}(2t+1))\\
% \frac12 (1-e^{-2t}(2t+1)) &  1-e^{-2t} \\
% \end{array}
% \right)\\
% &+ \left( \begin{array}{cc}
% e^{-2t}\mathbb E(X_0+tY_0)^2  & e^{-2t} \mathbb E(X_0Y_0+tY_0^2)\\
% e^{-2t} \mathbb E(X_0Y_0+tY_0^2)  & e^{-2t}\mathbb E Y_0^2
% \end{array}
% \right)
% \end{align}
% As expected, as $t\rightarrow \infty$, $\Sigma_t \rightarrow \Sigma$ and 
% \begin{equation}
%     m_t= (e^{-t} \mathbb E X_0+te^{-t}\mathbb E Y_0, e^{-t}\mathbb E Y_0)^T  \rightarrow (0,0)^T.
% \end{equation}
The $x$-marginal $\bar\rho_t $ of $\rho_t$ is given by
\[
\bar\rho_t = \cN(m_t^x, \Sigma_t^{xx}) =  \cN(e^{-t} m_0^x+te^{-t}m_0^y, \Sigma_t^{xx} ) \,.
\]
Similarly, the marginal $\bar\rho$ of $\rho$ is $\bar \rho = \cN (0, \Sigma^{xx})=\cN(0,1/2)$. 
Hence, $\bar{\rho}_t \rightarrow \bar\rho$ as $t\rightarrow \infty$. 

For the dynamics \eqref{example:Carsten'sexample1}-\eqref{example:Carsten'sexample2} we now want to calculate both the approximation $X_G$ produced via the Gy\"ongy method and the CG version $X_P$ produced via the projection method. To this end we need to compute the conditional distributions $\rho_t(\cdot\vert x) $ and $\rho(\cdot\vert x)$. 
The conditional distribution $\rho_t(\cdot \vert x) $ is 
$$
\rho_t(\cdot\vert x) = \mathcal N(m_t^{y\vert x}, \Sigma_t^{y\vert x}), 
$$
with (e.g.~see \cite[Sec.~8]{matrixCookbook})
\begin{align}
 m_t^{y\vert x} &= m_t^y+ \Sigma_t^{xy}(\Sigma_t^{xx})^{-1}(x-m_t^x)\\
 \Sigma_t^{y\vert x} &= 
 \Sigma_t^{yy} - \Sigma_t^{xy}(\Sigma_t^{xx})^{-1} \Sigma_t^{xy}
\end{align}
giving  
\begin{align}
 m_t^{y\vert x} &= e^{-t}m_0^y+ \frac{1 - e^{-2t}\left(2t+1\right) + 2t e^{-2t}\Sigma^{yy}_0}{1 - e^{-2t}\left(2t^2+2t+1 - 2\Sigma^{xx}_0 - 2t^2\Sigma^{yy}_0\right)}(x-m_t^x)\\
 \Sigma_t^{y\vert x} & = 
 \Sigma_t^{yy} - \Sigma_t^{xy}(\Sigma_t^{xx})^{-1} \Sigma_t^{xy} \,.
\end{align}
% \begin{align}
%  m_t^{y\vert x} &= e^{-t}m_0^y+ \frac{[ (1-e^{-2t}(2t+1 -2t \mathbb E Y_0^2)) ]}{\left(e^{-2t} (-2t^2-2t-1)+1 \right) + 2e^{-2t}(\mathbb E X_0^2+t^2 \mathbb E Y_0^2)}(x-m_t^x)\\
%  \sigma_t^{y\vert x} &= 
%  \Sigma_t^{yy} - \Sigma_t^{xy}(\Sigma_t^{xx})^{-1} \Sigma_t^{xy} \,.
% \end{align}
Using analogous formulas for $\rho=\mathcal N (0, \Sigma)$, we observe that  
$\rho(\cdot\vert x) \sim \mathcal N(x, 1/2)$, which implies that $\rho_t(\cdot\vert x) \rightarrow \rho(\cdot\vert x)$ as $t \rightarrow \infty$.

Finally, since in our case $f(x,y)=-x+y$ and $\alpha(x,y)=0$, using the expressions \eqref{eqn:projection-coefficients}-\eqref{eqn:projection-coefficients2} and \eqref{eqn:gyongy-averaged-coefficients}-\eqref{eqn:gyongy-averaged-coefficients2}, we obtain $\alpha_P=\alpha_G=0$ and 
$$
f_P(x)= \int_{\R}(-x+y) \rho(dy\vert x) = -x+x=0
$$
 and 
$$
 f_G(t,x) = -x+m_t^{y\vert x} \,.
$$
As a consequence, the projected equation reads  
 \begin{equation}\label{eq:counterexampleXP}
     dX_P=0 \,, \quad X_P(0)=X_0\,, 
 \end{equation}
with the trivial solution $X_P(t)\equiv X_0$ for all $t\ge 0$. 
Clearly, the result of Proposition \ref{lem:bar rho invariant for Lp} still holds for this example (even if the generator $\cL$ is not elliptic). In particular, the calculation in the block of equations after  \eqref{eqn:joint=conditional times marginal} still works for this process. This is  symptomatic of the limitation of Proposition \ref{lem:bar rho invariant for Lp}, which only states that if $\rho$ is invariant for $\cL$ then its marginal $\bar \rho$ is invariant for $\cL_P$. This is necessarily true in this case, as $\cL_P=0$, so any initial distribution is invariant for $\cL_P$. However $\cL_P$ has infinitely many invariant measures, while $\cL$ has only one. In other words, coarse graining by means of the projection method has in this case changed the number of stationary states.

% \begin{figure}
%     \centering
%     \includegraphics[width=0.65\linewidth]{figs/gyoengyCounterexplPhi.eps}
%     \caption{Function $\phi$ as defined in (\ref{eq:counterexampleXG})--(\ref{eq:counterexampleXGphi}).}
%     \label{fig:enter-label}
% \end{figure}

This should be contrasted with the Gy\"ongy approximation that reads 
 \begin{equation}\label{eq:counterexampleXG}
     \frac{d X_G}{dt} = (\phi(t)-1) X_G + e^{-t}m_0^y - \phi(t)\left(e^{-t} m_0^x+te^{-t}m_0^y\right)\,,
 \end{equation}
 where  
\begin{equation}\label{eq:counterexampleXGphi}
    \phi(t) = \frac{1 - e^{-2t}\left(2t+1\right) + 2t e^{-2t}\Sigma^{yy}_0}{1 - e^{-2t}\left(2t^2+2t+1 - 2\Sigma^{xx}_0 - 2t^2\Sigma^{yy}_0\right)} \,.
\end{equation}

If $\Sigma_0^{xx}>0$ the function $\phi$ has the property that $\phi(t)\to 1$ as $t\to\infty$,  $\phi(t)\to 0$ as $t\rightarrow 0$, and it is smooth and bounded. When $\Sigma_0^{xx}=0$ then the limit   as  $t\rightarrow 0$ of $\phi(t)$ depends on $\Sigma_0^{yy}$.   If $\Sigma_0^{yy}=0$ then $\phi(t)\to + \infty$, if $\Sigma_0^{yy}>0$ then $\phi(t)\to - \infty$. 

Bottom line,  if $\Sigma_0^{xx}>0$, the drift in (\ref{eq:counterexampleXG}) is bounded and continuous in $t$ and satisfies a global Lipshitz condition in $x$, which implies that the equation for  
$X_G$ is  well-posed. However, for  initial data that are deterministic in the $x$ component, the drift (\ref{eq:counterexampleXG})  blows up as $t$ tends to zero, so the ODE for the Gy\"ongy approximation is not well posed near the origin.  %{\color{red} however we should be careful in phrasing this, because if $\Sigma_0^{xx}=0$ then the ODE for gyongy CG is not well posed as $t$ tends to zero} 
As we have pointed out, even if it was globally well posed, the law of its solution would still be a Dirac and hence could not coincide with the law of $X_t$.

The question remains of whether, at least in the case  $\Sigma_0^{xx}>0$ (random initial conditions),   $X_G$ captures the correct dynamics. In Figure \ref{fig:exampleGyongyMeanVar} and Figure \ref{fig:exampleGyongyRhobar} 
we investigate this numerically. Figure \ref{fig:exampleGyongyMeanVar} suggests that, at least for this example, 
 the law of the solution of the Gy\"ongy approximation agrees with the marginal  distribution  $\bar{\rho}_t$ for finite $t>0$. Figure \ref{fig:exampleGyongyRhobar} looks at the long time behaviour of $X_G$, which again is the correct one.

Finally, still in the case $\Sigma_0^{xx}>0$, it makes sense to ask whether $X_G(t)$ and $X_P(t)$ have the same long time behaviour. Figure \ref{fig:exampleGyongy} shows that this is not the case. This is not surprising as the conditions of Theorem \ref{thm:appendixnonautonomous} are not satisfied in this example,  ellipticity being the condition that most obviously fails. 

\begin{figure}
    \centering
        \includegraphics[width=0.65\linewidth]{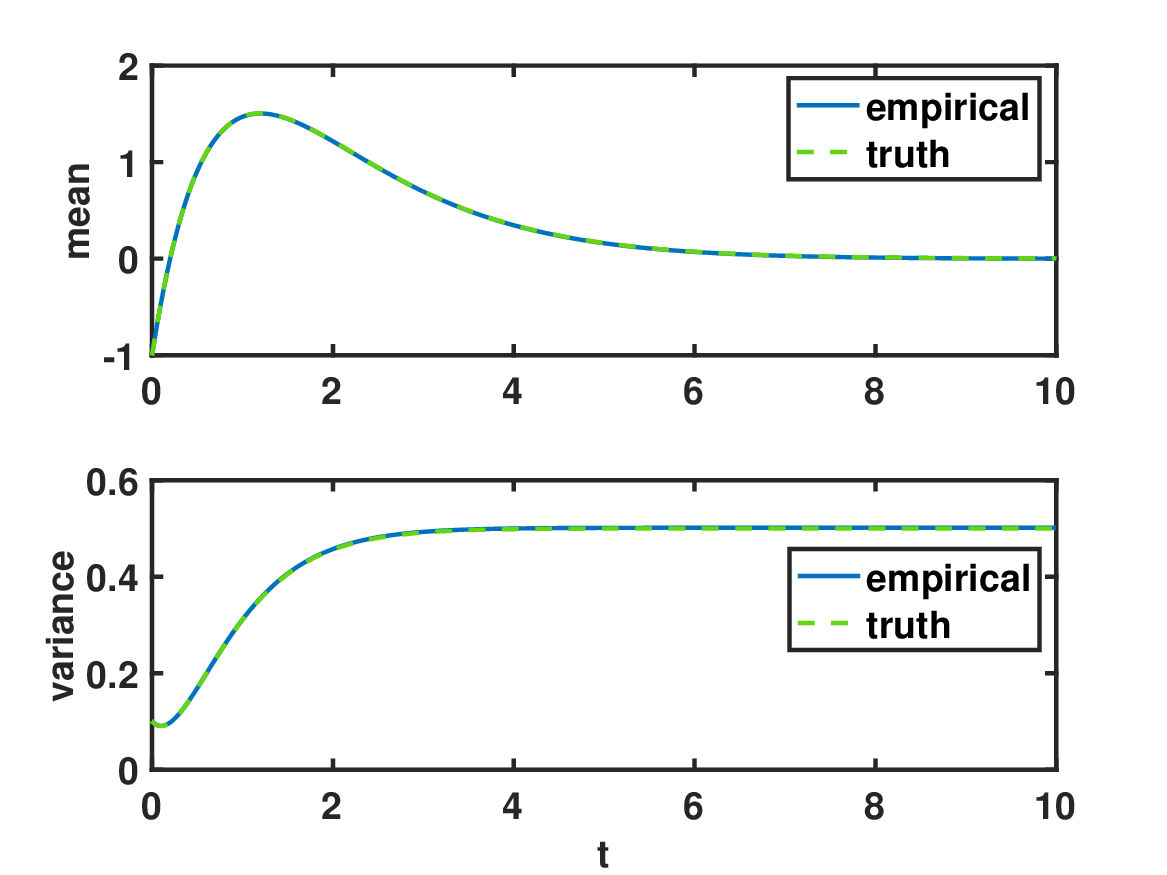}
    \caption{Behaviour of Gyöngy approximation \eqref{eq:counterexampleXG} of the 2-dimensional system (\ref{example:Carsten'sexample1}) for Gaussian initial conditions $X_0\sim\cN(-1,0.1)$ and $Y_0\sim\cN(5,1)$. The plotted mean and variance have been computed by averaging over trajectories starting from $N=10^5$ independent initial conditions $X_0$.}
    \label{fig:exampleGyongyMeanVar}
\end{figure}

All simulations have been carried out using an explicit 8(9)th-order Runge-Kutta pair with interpolation \cite{verner2010numerically}.

\begin{figure}
    \centering
        \includegraphics[width=0.5\linewidth]{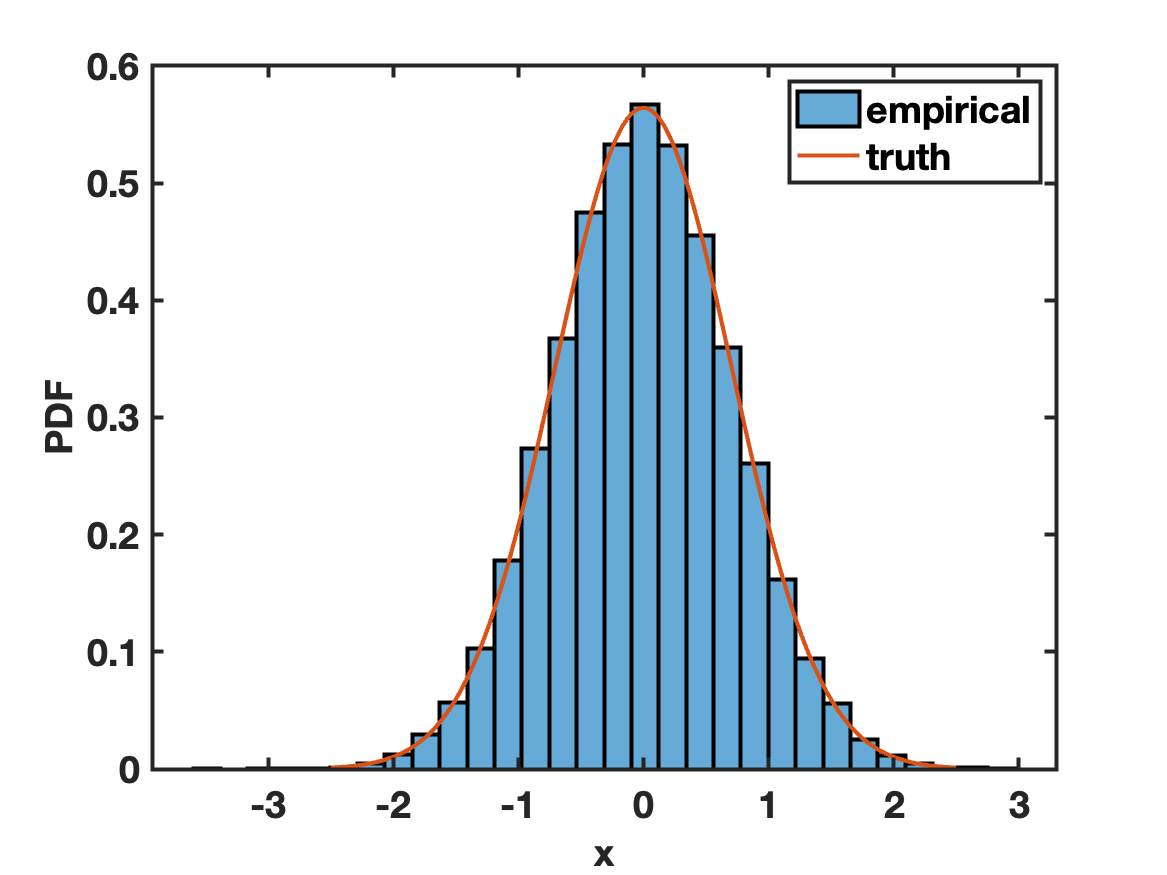}
     \caption{Long term dynamics of the Gyöngy approximation \eqref{eq:counterexampleXG}: The histogram was computed by binning $N=10^5$ trajectories starting from independent  random initial data $X_0\sim\bar{\rho}_0\neq \cN(0,1/2)$ and  evaluated at $t=20$; the red curve shows $\bar{\rho}=\cN(0,1/2)$.}
    \label{fig:exampleGyongyRhobar}
\end{figure}

\begin{figure}
    \centering
    \includegraphics[width=0.495\linewidth]{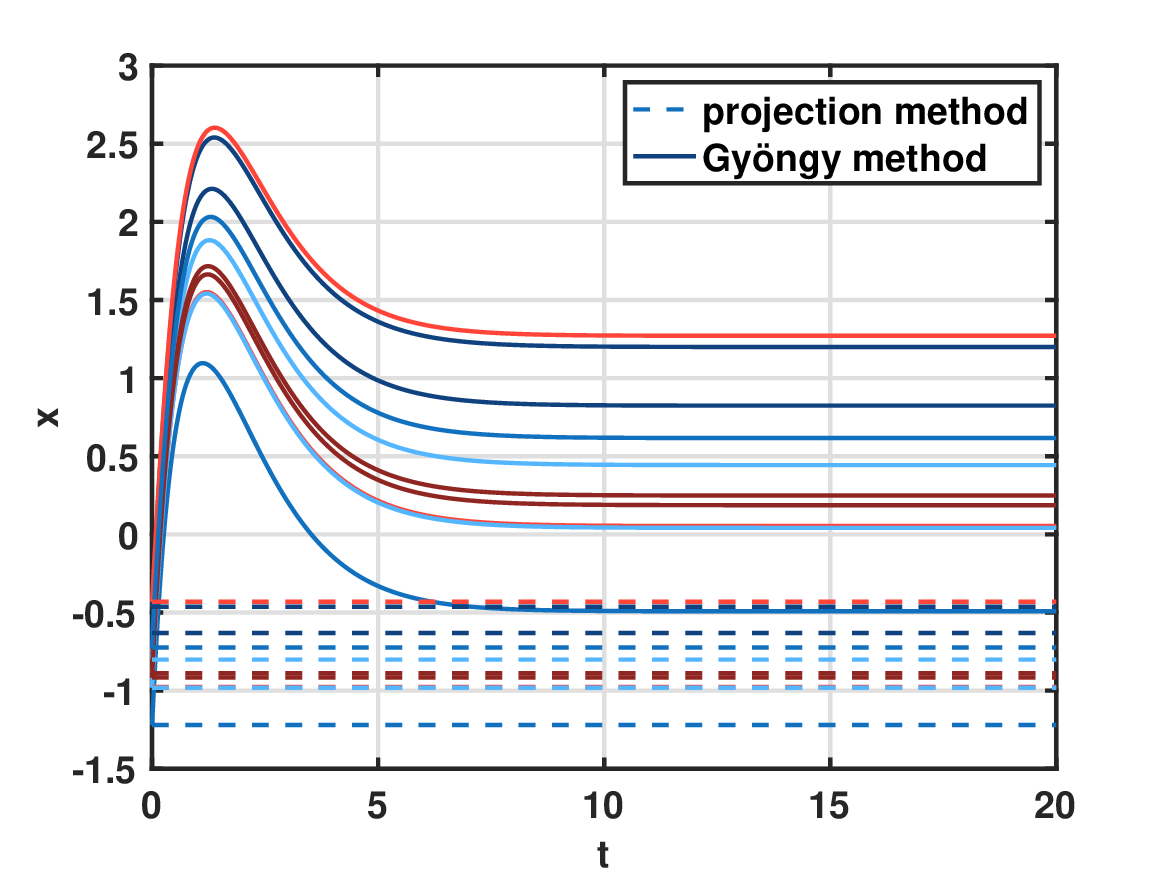}
    \includegraphics[width=0.495\linewidth]{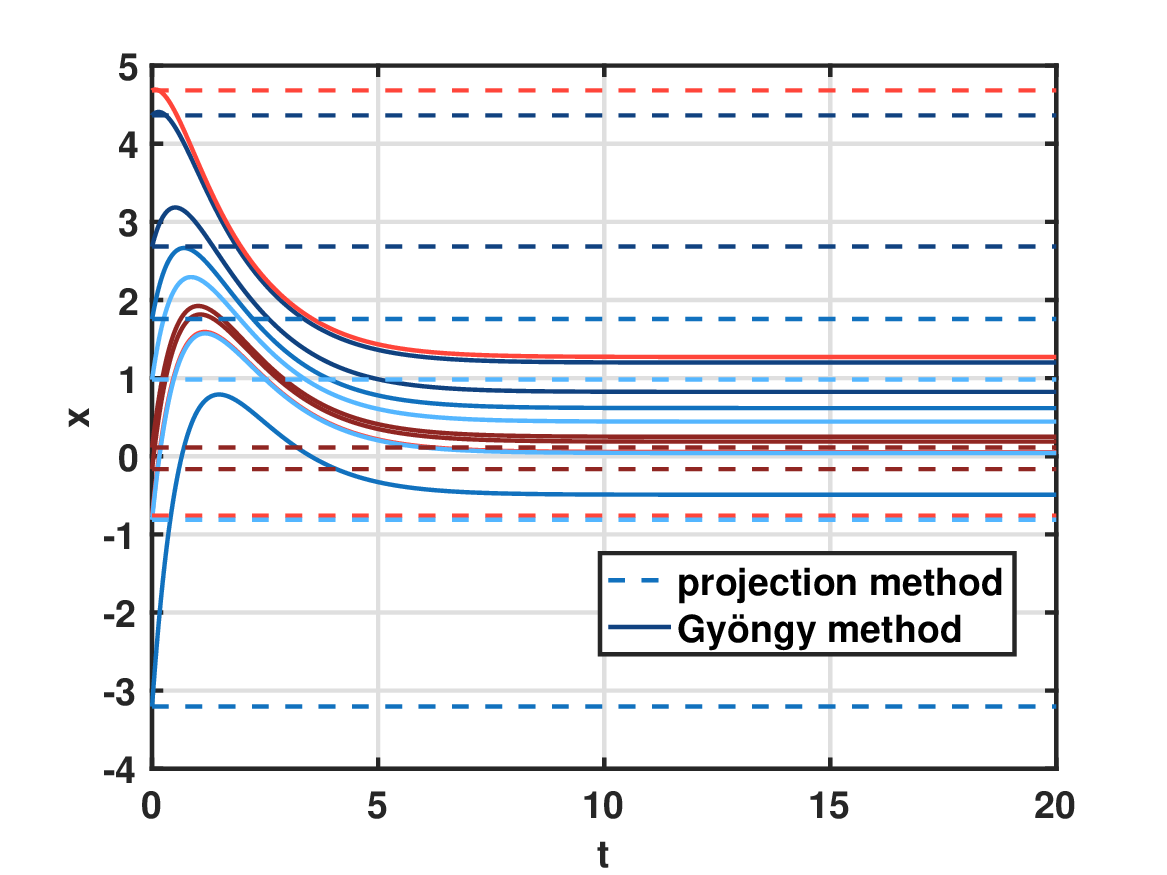}
     \caption{Solution of (\ref{eq:counterexampleXP}) and (\ref{eq:counterexampleXG}) for 10 representative initial conditions  $X_0\sim \cN(-1,0.1)$ (left panel) and $X_0\sim \cN(-1,10)$ (right panel). Both figures show convergence of the Gyöngy trajectories (solid lines) to an ensemble with the correct asymptotic mean and variance on average, whereas  the solutions of the projected dynamics (dashed lines) stay constant. (The same colour indicates that the solutions start from the same initial conditions.) In particular the trajectories of $X_G$ do not converge to $X_P(t)=X(0)$ for $t$ large.  The trajectories nicely show the inflation or deflation of trajectories,  depending in the whether $\Sigma_0^{xx}$ is smaller or larger than the asymptotic variance $\Sigma^{xx}_\infty$.}
      \label{fig:exampleGyongy}
\end{figure}

% {\color{red} Michela: can someone more skilled with matlab than me solve the ODE for $X_G$ in the above.  I don't think that as $t\rightarrow \infty$ the solution of $X_G(t)$ tends to $X_0$. This should happen if Theorem 3.4 did hold for those two dynamics, but I don't think it does so I don't expect that $X_G$ and $X_p$  have the same asymptotic behaviour. If this turns out to be the case I can explain in a note why.  \\
% Finally, can someone double check that you agree with me on the well-posedness of the equation for $X_G$. There are some observations at the end of the notes that Carsten sent some time ago on this example that I don't fully understand, hence I am a little worried I may be missing something. }

%%%%%%%%%%%%%%%%%%%%%%%%5
%%%%%%%%%%%%%%%%%%%%%%%%%%%%%%%
%%%%%%%%%%%%%%%%%%%%%%%%%%%%%%%%
%%%%%%%%%%%%%%%%%%%%%%%%%%%%%55

%%%%%%%%5
\appendix
\section{Long time behaviour of non-autonomous SDEs}\label{appendix:long time behaviour of non-autonomous SDEs}

\begin{proof}[Proof of Theorem \ref{thm:appendixnonautonomous}.] This is a consequence of \cite[Theorem 6.5]{cass2021long}. Notice that the statement of this theorem assumes  that $\sigma$ is globally Lipshitz (no boundedness required) and also allows to consider space-dependent $\sigma$. 

Aside from this fact, \cite[Theorem 6.5]{cass2021long} holds under four main assumptions, listed in   \cite[Hypothesis 6.1, H.1 to H.4]{cass2021long}. We first show that our Assumption \ref{ass:Lunardiappendix} \ref{ass:unformellipticity appendix}, \ref{ass:Lyapcondappendix}, \ref{ass:ObtuseAngleConditionappendix} imply [H.1], [H.2], [H.3] in \cite[Hypothesis 6.1]{cass2021long}, respectively. We then show that the proof of \cite[Theorem 6.5]{cass2021long} still works if \cite[Hypothesis 6.1, H.4]{cass2021long} is replaced by our Assumption \ref{ass:Lunardiappendix} \ref{ass:coefficients have limit appendix}. 

We start by noting that \cite[Theorem 6.5]{cass2021long} refers to dynamics of the form 
\begin{align}
    & d\mathcal X_t =  B(\zeta_t, \mathcal X_t) dt + \sum_{i=1}^d\sigma\col (\zeta_t, \cX_t) \circ dU^i_t \label{eqn:SDEappendix}\\
    & d\zeta_t = A(\zeta_t) dt \label{eqn:ODEappendix}  \,,
\end{align}
    where $A\colon\R\rightarrow \R$ so that \eqref{eqn:ODEappendix} is just a one-dimensional ODE and \eqref{stratonovich non-autonomous general SDE} is recovered by just choosing $A(\zeta)\equiv 1$. There are various comments in \cite{cass2021long} on the relation between the system \eqref{eqn:SDEappendix}-\eqref{eqn:ODEappendix} and the dynamics \eqref{stratonovich non-autonomous general SDE}, we do not repeat them here. 

\cite[Hypothesis 6.1, H.1]{cass2021long} is implied by our Assumption \ref{ass:Lunardiappendix} \ref{ass:unformellipticity appendix}. To see this just note that the vector fields that in \cite[Hypothesis 6.1, H.1]{cass2021long} are denoted by $V_0=(U_0, W_0), V_1=(U_1, 0),\dots ,V_d=(U_d,0)$ are in our case $V_0=(B, A)$ and $V_i = (\sigma\col,0)$ for $i=1, \dots, d$. Since $\sigma$ is uniformly elliptic, it is easy to see that the vectors $V_0, V_1, \dots, V_d$ satisfy the UFG condition. We do not recall the UFG condition here, as stating such a condition requires some amount of notation, but we refer the reader to  \cite[Definition 3.1]{cass2021long} for a precise statement and observe that such a condition  is satisfied in our case  with $m=1$. In other words, the set that in that paper is called $\mathcal A$ in this simple case is just the set $\mathcal A= \{1, \dots, d\}$.  

\cite[Hypothesis 6.1, H.2]{cass2021long} is implied by our Assumption \ref{ass:Lunardiappendix} \ref{ass:Lyapcondappendix}, see \cite[Note 6.2]{cass2021long}, fourth bullet point in that note. 

\cite[Hypothesis 6.1, H.3]{cass2021long} is implied by our Assumption \ref{ass:Lunardiappendix} \ref{ass:ObtuseAngleConditionappendix}. This is easy to check by using the above remarks and the identity \cite[Equation (6.9)]{cass2021long}. 

All is left to show is that \cite[Hypothesis 6.1, H.4]{cass2021long} can be replaced by our Assumption \ref{ass:Lunardiappendix} \ref{ass:coefficients have limit appendix}. \cite[Hypothesis 6.1, H.4]{cass2021long} requires $\zeta_t$ solution of the ODE  \eqref{eqn:ODEappendix} to have a limit. Informally, note that  if $\lim_{t\rightarrow \infty} \zeta_t  = \bar\zeta \in \R$ then, since $B$ is smooth,  clearly $\lim_{t\rightarrow \infty}B(\zeta_t, x) = B(\bar\zeta, x) = \bar B(x)$ i.e. the drift coefficient admits a limit (and similarly for the diffusion coefficient). So, at least intuitively, it should be clear that these two assumptions achieve a similar goal.  We however strictly speaking cannot say that Assumption \ref{ass:Lunardiappendix} \ref{ass:coefficients have limit appendix}  implies \cite[Hypothesis 6.1, H.4]{cass2021long}. The latter was phrased in that way because in \cite{cass2021long} the authors were interested not only the in dynamics $\cX_t$ but also in the full system \eqref{eqn:SDEappendix}-\eqref{eqn:ODEappendix} and such a system is not ergodic if one chooses $A\equiv 1$. Nonetheless, since we are simply interested in \eqref{eqn:SDEappendix}, our  Assumption \ref{ass:Lunardiappendix} \ref{ass:coefficients have limit appendix} suffices to achieve the result and indeed the proof of \cite[Theorem 6.5]{cass2021long} goes through also under this assumption. We do not repeat the whole proof, as it is lengthy but point out that the main place that could potentially lead to difficulties is in the section of that proof that is deferred to the proof of \cite[Lemma B.1]{cass2021long}, reasoning after (B.13). The important thing to notice is that the estimate above (B.14) can still be achieved because the validity of  the estimate (19) for the semigroup there denoted by $\mathcal P_{t}$ is ensured by our Assumption \ref{ass:Lunardiappendix} \ref{ass:ObtuseAngleConditionappendix}. 
\end{proof}

\begin{lemma}\label{lemma B3}
    If $A$ is Hurwitz (i.e. all eigenvalues have strictly negative real part), and the matrix pair $(A,C)$ satisfies the Kalman rank condition,
    \begin{equation}
        {\rm rank}\Big(C|AC|A^2C|\ldots|A^{n-1}C\Big)=n\,,
    \end{equation}
    then the OU process $dZ_t=AZ_tdt+Cd\bf{W}_t$ has a unique stationary Gaussian distribution $\mathcal{N}(0,\Sigma)$ that is reached asymptotically as $t\to\infty$.  Moreover $\Sigma$ is symmetric positive definite (s.p.d.) and satisfies the linear Lyapunov equation
    \begin{equation}\label{eq:lyapunov}
        A \Sigma + \Sigma A^T  = -CC^T \,.
    \end{equation}
\end{lemma}
\begin{proof}
    The convergence $\bE[Z_t]\to 0$ is a direct consequence of the Hurwitz property. Setting $\Sigma_t:={\rm Cov}(Z_t,Z_t)$ and using that $\exp(At)\Sigma_0\exp(A^Tt)\to 0$ as $t\to\infty$, the matrix
    \begin{equation*}
        \Sigma :=  \lim_{t\to\infty}\int_0^t \exp(A(t-v)) CC^T \exp(-A^T(t-v))\,dv= \int_0^\infty \exp(Av)CC^T \exp(A^T v)\,dv
    \end{equation*}
    exists. As a consequence of the Caley-Hamilton Theorem (e.g.~\cite{axler1997}), a vector $h\neq 0$ lies in the kernel of $(e^{-At}C)^T $ for any $t>0$ iff $h^T A^{k}C=0$ for all $k=0,\ldots n-1$. This implies that $\Sigma_\infty>0$ if the Kalman condition holds. Integrating the rightmost integral by parts and using that $A$ is invertible, yields the Lyapunov equation.
\end{proof}
\section*{Acknowledgment}
The authors acknowledge the support from the International Centre for Mathematical Sciences (ICMS) through a Research In Group Grant and through an  ICMS-funded workshop,  where this project started. The research of H.D. was funded by the EPSRC Grant EP/Y008561/1. M.O. gratefully acknowledges support from a Research Fellowship of the Royal Society of Edinburgh RSE3789. The research of C.H. was paertially funded by the Deutsche
Forschungsgemeinschaft (DFG) through the grant CRC 1114: Scaling Cascades in Complex Systems (project no. 235221301) and the
Federal Ministry of Education and Research and the State of Brandenburg within the framework of the joint project EIZ: Energy Innovation Center (project numbers 85056897 and 03SF0693A).

\end{document}